\documentclass[11p]{article}

\usepackage[numbers,sort]{natbib}
\bibpunct[, ]{[}{]}{,}{n}{}{,}%

\usepackage[utf8]{inputenc}
\usepackage{svg}
\usepackage{amsmath}
\usepackage{amsfonts}
\usepackage{amsthm}
\usepackage{hyperref}
\usepackage{mathtools}
\usepackage{amssymb}

\DeclarePairedDelimiter\floor{\lfloor}{\rfloor}
\usepackage{breakurl}
\usepackage{cleveref}

\usepackage[english]{babel}
\usepackage[utf8]{inputenc}
\usepackage{graphicx}
\usepackage{xcolor}
\usepackage{color}
\renewcommand{\vec}[1]{\mathbf{#1}}
\usepackage{mathrsfs}
\usepackage[paper=a4paper,margin=1in]{geometry}
\usepackage[nottoc]{tocbibind}
\usepackage[numbers]{natbib}
\usepackage[colorinlistoftodos]{todonotes}
\usepackage[normalem]{ulem}
\DeclareMathOperator{\Tr}{Tr}
\newtheorem{theorem}{Theorem}
\newtheorem{corollary}{Corollary}[theorem]
\newtheorem{lemma}[theorem]{Lemma}
\newtheorem{proposition}[theorem]{Proposition}
\newtheorem{conjecture}{Conjecture}
\newtheorem{definition}{Definition}

\newcommand{\GraphName}{\mathcal{G}_n}
\newcommand{\Setname}{f^+}
\begin{document}

\title{On the generalized  $\vartheta$-number and related problems for highly symmetric graphs}
\author{Lennart Sinjorgo \thanks{CentER, Tilburg University, The Netherlands, {\tt l.m.sinjorgo@tilburguniversity.edu}}
	\and {Renata Sotirov}  \thanks{Department of Econometrics and OR, Tilburg University, The Netherlands, {\tt r.sotirov@uvt.nl}}}

\date{}

\maketitle

\begin{abstract}
This paper is an in-depth analysis of the generalized $\vartheta$-number  of a graph.
The generalized $\vartheta$-number, $\vartheta_k(G)$, serves as a bound for both the $k$-multichromatic number of a graph and the maximum $k$-colorable subgraph problem. We present various properties of $\vartheta_k(G)$,  such as that the sequence  $(\vartheta_k(G))_k$ is increasing and bounded from above by the order of the graph $G$.  We study  $\vartheta_k(G)$ when $G$ is the strong, disjunction or Cartesian product of two graphs. We provide closed form expressions for the generalized $\vartheta$-number on several classes of graphs including the Kneser graphs, cycle graphs,   strongly regular graphs and orthogonality graphs.
Our paper provides  bounds on the product and sum of the $k$-multichromatic number of a graph and its complement graph,  as well as lower bounds for the $k$-multichromatic number on several graph classes including the Hamming  and Johnson graphs.
\end{abstract}



\section{Introduction}
A $k$--multicoloring of a graph is an assignment of $k$ distinct colors to each vertex in the graph such that two adjacent vertices are assigned  disjoint sets of colors. The $k$-multicoloring is also known as $k$-fold coloring, $k$-tuple coloring or simply multicoloring. We denote by $\chi_k(G)$  the minimum number of colors needed for a valid $k$--multicoloring of a graph $G$, and refer to it as the $k$-th chromatic number of $G$ or the multichromatic number of $G$.
Multicoloring seems to have been independently introduced by   \citeauthor{hilton19735} \cite{hilton19735} and \citeauthor{stahl1976n} \cite{stahl1976n}. The $k$--multicoloring is a generalization of the well known standard graph coloring. Namely,  $\chi(G):=\chi_1(G)$ is known as the chromatic number of a graph $G$. Not  surprisingly, multicoloring finds applications in comparable areas, such as job scheduling  \cite{gandhi2004improved,halldorsson2004multicoloring}, channel assignment in cellular networks \cite{narayanan2002channel} and register allocation in computers \cite{chaitin1982register}. There exist several results on $\chi_k(G)$ for specific classes of graphs. In particular,  \citeauthor{lin2008multicoloring} \cite{lin2008multicoloring} and \citeauthor{lin2010multi} \cite{lin2010multi} consider  multicoloring the Mycielskian of  graphs,
\citeauthor{ren2010k} \cite{ren2010k} study multicoloring of planar graphs while \citeauthor{marx2002complexity} \cite{marx2002complexity} proves that the multicoloring problem is strongly NP-hard in binary trees. \citeauthor{cranston2018planar} \cite{cranston2018planar} show that, for any planar graph $G$, $\chi_2(G) \leq 9$. This result is implied by the famous four-color theorem by \citeauthor{appel1977every} \cite{appel1977every}, but it has a much simpler proof.

The maximum $k$-colorable subgraph (M$k$CS) problem  is to find  the largest induced subgraph in a given graph that can be colored with $k$ colors so that no two adjacent vertices have the same color. When $k = 1$, the M$k$CS problem reduces to the well known  maximum stable set problem.
The M$k$CS problem is one of the  NP-complete problems considered by~\citeauthor{ksubgrNP} \cite{ksubgrNP}. We denote by $\alpha_k(G)$ the number of vertices in a maximum $k$-colorable subgraph of $G$, and by  $\omega_k(G)$ the size of the largest  induced subgraph that can be partitioned into $k$ cliques. When  $k=1$, the graph parameter   $\omega(G):=\omega_1(G)$ is known as the clique number of a graph,  and  $\alpha(G):=\alpha_1(G)$ as the independence number of a graph.
We note that  $\alpha_k(G) = \omega_k(\overline{G})$, where $\overline{G}$ denotes the complement of $G$. The  M$k$CS problem has  a number of applications such as channel assignment in spectrum sharing networks~\cite{KosterScheffel,networksApp2},  VLSI design \cite{VLSI1} and human genetic research~\cite{VLSI2,biology}. There exist several results on $\alpha_k(G)$ for specific classes of graphs. The size of the maximum $k$-colorable subgraph for the Kneser graph $K(v, 2)$ is provided by \citeauthor{FurediFrankl} \cite{FurediFrankl}. \citeauthor{chordal1} \cite{chordal1} consider the  M$k$CS problem for chordal graphs,  \citeauthor{iTriang} \cite{iTriang} study the problem for an $i-$triangulated graph, and \citeauthor{colorThesis} \cite{colorThesis} computes $\alpha_k(G)$ for  circular-arc graphs and tolerance graphs.

\citeauthor{narasimhan1988generalization} \cite{narasimhan1988generalization} introduce a graph parameter $\vartheta_k(G)$ that serves as a bound for both the minimum number of colors needed for a $k$--multicoloring of a graph $G$ and the number of vertices in a maximum $k$-colorable subgraph of $G$. The parameter $\vartheta_k(G)$   generalizes the concept of the  famous $\vartheta$-number  that  was introduced by   \citeauthor{lovasz1979shannon} \cite{lovasz1979shannon} for bounding the Shannon capacity of a graph  \cite{Shannon1956TheZE}. The  Lov{\'a}sz theta number is a  widely studied graph parameter see e.g., \cite{chan2009sdp,grotschel1981ellipsoid,gvozdenovic2008operator,knuth1994sandwich,LovaszBook,mceliece1978lovasz}. The Lov{\'a}sz theta number  provides bounds for both the clique number and the chromatic number of a graph, both of which are NP-hard to compute.  The well known result that establishes the following  relation $\alpha_1(G) \leq \vartheta_1(G) \leq \chi_1 (\overline{G})$ or equivalently $\omega_1(G) \leq \vartheta_1(\overline{G}) \leq \chi_1 (G)$  is known as the sandwich theorem \cite{LovaszBook}. The Lov{\'a}sz theta number can be computed in polynomial time as an semidefinite programming (SDP) problem by using interior point methods.
Thus, when the clique number and chromatic number of a graph coincide i.e., when the graph is weakly perfect, the Lov{\'a}sz theta number provides those quantities in polynomial time.

Despite the popularity of the Lov{\'a}sz theta number, the function $\vartheta_k(G)$ has received little attention in the literature. \citeauthor{narasimhan1988generalization} \cite{narasimhan1988generalization} show that $\alpha_k(G) \leq \vartheta_k(G) \leq \chi_k (\overline{G})$ or equivalently  $\omega_k(G) \leq \vartheta_k(\overline{G}) \leq \chi_k (G)$.
These inequalities can be seen as a generalization of the Lov\'asz sandwich theorem.
\citeauthor{Alizadeh} \cite{Alizadeh} formulates the generalized $\vartheta$-number using semidefinite programming. \citeauthor{kuryatnikova2020maximum} \cite{kuryatnikova2020maximum} introduce the generalized $\vartheta'$-number that is obtained by  adding non-negativity constraints to the SDP formulation of the  $\vartheta_k$-number. The generalized $\vartheta$-number and $\vartheta'$-number are evaluated numerically as upper bounds for the M$k$CS problem in \cite{kuryatnikova2020maximum}. The authors of \cite{kuryatnikova2020maximum} characterise a family of graphs for which  $\vartheta_k(G)$ and $\vartheta'_k(G)$ provide tight bounds for $\alpha_k(G)$.  Here, we  study also  a relation between  $\vartheta_k(G)$ and $\chi_k(G)$,
and extend many known results for the Lov{\'a}sz  $\vartheta$-number to the generalized  $\vartheta$-number.  This paper is based on the thesis of  \citeauthor{sinjorgo2021MsCthesis}  \cite{sinjorgo2021MsCthesis}.

\subsection*{Main results and outline }
This paper provides various theoretical results  for $\alpha_k(G)$,  $\vartheta_k(G)$ and  $\chi_k(G)$.
We show numerous properties of $\vartheta_k(G)$ including results on different graph products of two graphs such as the Cartesian product, strong product and disjunction product.
We show that the sequence  $(\vartheta_k(G))_k$ is increasing towards the number of vertices in $G$, and that the increments of the sequence  can be arbitrarily small. The latter result is proven by constructing a particular graph that satisfies the desired property.
We also provide a closed form expression on the generalized $\vartheta$-number for several graph classes including  complete graphs, cycle graphs, complete multipartite graphs, strongly regular graphs, orthogonality graphs, the Kneser graphs and some Johnson graphs.
We  compute  $\vartheta_k(G)$ for circulant graphs and the Johnson graphs. Our results show that $\vartheta_k(G)=k \vartheta(G)$ for  the Kneser graphs and more general Johnson graphs,  strongly regular graphs, cycle graphs and  circulant graphs.  Our paper presents lower bounds on  the $k$-th chromatic number for all regular graphs, but also specialized bounds for the Hamming, Johnson and orthogonality graphs. We also provide bounds on the product and sum of $\chi_k({G})$ and $\chi_k(\overline{G})$, and present graphs for which those bounds are attained. Those results generalize well known results of \citeauthor{nordhaus1956complementary} \cite{nordhaus1956complementary} for  $\chi(G)$ and $\chi(\overline{G})$.

This paper is organized as follows. Notation and definitions of several graphs and graph products are given in section \ref{sect:NotationDefiniton}.  In section \ref{section:formulations} we formally introduce $\vartheta_k(G)$  and $\chi_k(G)$ and show how those graph parameters relate. In section \ref{sect:the_series}  we study the sequence $(\vartheta_k(G))_k$. Section \ref{sect:Graph products}  provides bounds for $\vartheta_k(G)$ when $G$ is the strong graph product of two graphs  and the disjunction product of two graphs.  In section \ref{Section:ValueOfThetaKForSomeGraphs} one can find values of the generalized $\vartheta$-number for  complete  graphs,   cycle graphs,  circulant graphs  and complete multipartite graphs. In section \ref{section:RegularGraphs} we provide a closed form expression for the generalized $\vartheta$-function on the Kneser graphs,  as well as for the Johnson graphs. Section \ref{Section:AnalysisOfCart} relates $\vartheta(K_k \square G)$ and $\vartheta_k(G)$. We provide a closed form expression for the generalized $\vartheta$-function for strongly regular graphs in section \ref{section:SRG}. In the same section we also relate  the Schrijver's number $\vartheta'(K_k \square G)$ with $\vartheta_k(G)$ when $G$ is a strongly regular graph.
In section \ref{section:OrthogonalityGraph} we study a relation between the orthogonality graphs and here considered graph parameters. Section \ref{sect:lower_bound_for_multicoloring} provides new lower bounds on the $k$-th chromatic number for regular graphs and triangular graphs. We present several results for the multichromatic number of the Hamming graphs in section \ref{section:HammingGraphs}.

\subsection{Notation and definitions} \label{sect:NotationDefiniton}
Let $\mathbb{S}^n$ be the space of symmetric $n$ by $n$ matrices. For matrix $X \in \mathbb{S}^n$, we write $X \succeq 0$ when $X$ is positive semidefinite. Entries of matrix $X$ are given by $X_{ij}$. The trace inner product for symmetric matrices is denoted $\langle X, Y \rangle = \Tr(XY)$. The Kronecker product of matrices is denoted by $ X \otimes Y$. By abuse of notation, we use the same symbol for the tensor product of graphs.
The matrix of all ones is denoted by $J$, while the identity matrix is denoted by $I$. We sometimes use subscripts to indicate the size of a matrix. Denote by $\vec{0}$ and $\vec{1}$ the vector of all zeroes and ones respectively.

For any graph $G = (V(G),E(G))$, we denote its adjacency matrix  by $A_G$, or simply $A$ when the context is clear.  Similarly, we use $V$ and $E$ to denote the vertex and edge set of $G$ when it is clear from the context. We assume that $|V|=n$, unless stated differently. The Laplacian matrix of a graph $G$ is denoted by $L_G$.   The complement graph of $G$, denoted by $\overline{G}$, is  defined as the graph such that $A_{\overline{G}} + A_G = J - I.$

For the eigenvalues of $X \in \mathbb{S}^n$, we follow $\lambda_1(X) \geq \lambda_2(X) \geq \ldots \geq \lambda_n(X)$, and denote by $\sigma(A)$  the spectrum of matrix $A$. That is,
$ \sigma(A) = \{ \lambda_1(A), \ldots, \lambda_n(A) \}.$
We denote the set $\{1,\ldots,n\}$ by $[n]$.
The `diag' operator maps an $n\times n$ matrix to the $n$-vector given by its diagonal. The adjoint operator of `diag' is denoted by `Diag'. 

In the rest of this section we provide several definitions. The first definition introduces several graph products, while the remaining definitions introduce different classes of graphs.
\begin{definition}[Graph products]
\label{Def:GraphProducts}
An arbitrary graph product of graphs $G_1 = (V_1, E_1)$ and $G_2 = (V_2, E_2)$ is denoted by $G_1 * G_2$, having as vertex set the Cartesian product $V_1 \times V_2$. Table \ref{Table:GraphProducts} shows when vertices $(v_1, v_2)$ and $(u_1, u_2)$ are adjacent in $G_1 * G_2$, for the lexicographic, tensor, Cartesian, strong and disjunction (\citeauthor{abdo2014total} \cite{abdo2014total}) graph products.
\begin{table}[h]
\centering
\begin{tabular}{ll@{}c@{}rl}
Graph product & \multicolumn{3}{c}{$G_1 *
G_2$} & Condition for $((v_1, v_2),(u_1, u_2)) \in E(G_1 * G_2)$ \\ \hline
 Lexicographic & $G_1$ & $\circ$ & $G_2$ &
 $(v_1, u_1) \in E_1$ \text{ or } $[v_1 = u_1$ \text{ and } $(v_2, u_2) \in E_2]$ \rule{0pt}{3ex} \\
Tensor & $G_1$ & $\otimes$ & $G_2$  & $(v_1, u_1) \in E_1$ and $(v_2, u_2) \in E_2$ \rule{0pt}{3ex} \\
Cartesian &  $G_1$ & $\square$ & $G_2$ & $[v_1 = u_1$ and $(v_2, u_2) \in E_2]$ or $[v_2 = u_2 \text{ and } (v_1,u_1) \in E_1]$ \rule{0pt}{3ex}   \\
Strong &  $G_1$ & $\boxtimes$ & $G_2$ & $((v_1, v_2),(u_1, u_2)) \in E(G_1 \square G_2) \cup E(G_1 \otimes G_2)$ \rule{0pt}{3ex} \\
Disjunction &  $G_1$ & $\vee$ & $G_2$ & $(v_1, u_1) \in E_1$ or $(v_2, u_2) \in E_2$ \rule{0pt}{3ex}
\end{tabular}
\caption{Graph products}
\label{Table:GraphProducts}
\end{table}
\end{definition}

In order to define the Hamming graphs, we first state the definition of the Hamming distance.
\begin{definition}[Hamming distance]
\label{Def:HammingDistance}
For two integer valued vectors $\vec{u}$ and $\vec{v}$, the Hamming distance between them, denoted by $d(\vec{u}, \vec{v})$, is the number of positions in which their entries differ.
\end{definition}

\begin{definition}[Hamming graph]
\label{Def:HammingGraph}
The Hamming graph $H(n, q, F)$ for $n, q \in \mathbb{N}$ and $F \subset \mathbb{N}$ has as vertices all the unique elements in $(\mathbb{Z}/q\mathbb{Z})^n$.  In the Hamming graph, vertices $u$ and $v$ are adjacent if their  Hamming distance $d(u,v) \in F$.
\end{definition}
Many authors define the Hamming graphs only for $F = \{ 1 \}$.

\begin{definition}[Johnson graph]
\label{Def:JohnsonGraphs}
Let $n,m \in \mathbb{N}$, $1\leq m \leq n/2$, $f \in \{0, 1, \ldots, m\}$ and  $N = [n]$. The Johnson graph $J(n,m,f)$ has as vertices all the possible $m$-sized subsets of $N$. Denote the subset corresponding to a vertex $u$ by $s(u)$. Then $|s(u)| = m$ and vertices $u$ and $v$ are adjacent if and only if $|s(u) \cap s(v)| = f$.
\end{definition}

Many authors define the Johnson graph only for $f =  m-1$. When $f =  0$, the Johnson graph is better known as the Kneser graph.

\begin{definition}[Kneser graph]
\label{Def:KneserGraph}
Let $n,m \in \mathbb{N}$ and $1\leq m \leq n/2$. Then the Kneser graph $K(n,m)$ is the Johnson graph $J(n,m,0)$.
\end{definition}

\begin{definition}[Strongly regular graph]
\label{Def:StronglyRegular}
A  $d$-regular graph $G$ of order $n$ is called  strongly regular with parameters $(n, d, \lambda, \mu)$ if any two adjacent vertices share $\lambda$ common neighbors and any two non-adjacent vertices share $\mu$ common neighbors.
\end{definition}

\section{$\vartheta_k(G)$  and $\chi_k(G)$  formulations and their relation} \label{section:formulations}

In this section, we formally introduce the  multichromatic number   and the generalized $\vartheta$-number of a graph. We also show a relationship between these two graph parameters.

Let $G = (V, E)$ be a simple undirected graph with $n$ vertices.
A valid $k$-multicoloring of $G$ that uses  $R$ colors is a mapping
$f: V \rightarrow 2^{R}$, such that $|f(i)| = k$ for all vertices $i\in V$ and $|f(i) \cap f(j)| = 0$ for all edges $(i,j) \in E$.
The  multichromatic  number $\chi_k(G)$ is defined to be the size of a smallest $R$ such that a valid $k$-multicoloring of $G$ exists.
Here we consider only valid $k$-multicoloring and refer to it as  $k$-multicoloring.

Multicoloring can be reduced to standard graph coloring by use of the lexicographic product of graphs, see
definition \ref{Def:GraphProducts}.
Namely, \citeauthor{stahl1976n} \cite{stahl1976n}  showed that for any graph $H$ such that $\chi(H) = k$, we have $\chi_k(G) = \chi(G \circ H)$. For clarity purposes, the simplest choice for $H$ is $K_k$, the complete graph of order $k$. This results
in
\begin{equation}
    \label{eqn:MultichromTransformation}
    \chi_k(G) = \chi(G \circ K_k).
\end{equation}

For bounds on the chromatic number of (lexicographic) graph products, we refer readers to \citeauthor{klavzar1996coloring} \cite{klavzar1996coloring} and \citeauthor{geller1975chromatic} \cite{geller1975chromatic}.
By the lexicographic product, any bound on $\chi(G)$ can also be transformed to a bound on $\chi_k(G)$. In particular:
\begin{equation}
    \label{eqn:ThetaKLouwerBound}
    \chi_k(G) = \chi(G \circ K_k) \geq \omega (G \circ K_k) = \omega(G) \omega(K_k) = k \omega(G).
\end{equation}
Here we use that $\omega(G \circ H) = \omega(G) \omega(H)$ for general graphs $G$ and $H$. We also mention the following result
\begin{equation}
    \label{eqn:IndepNumberLexiProd}
    \alpha(G \circ H) = \alpha(G) \alpha(H).
\end{equation}
Both results  are proven by  \citeauthor{geller1975chromatic} \cite{geller1975chromatic}.
Let us also state the following known result:
\begin{equation}
    \label{eqn:ThetaK_upperbound}
    \chi(G \circ H) \leq \chi(G) \chi(H).
\end{equation}
This result can be explained as follows. Denote the vertex sets of $G$ and $H$  by $V(G)$ and $V(H)$, respectively. For an optimal coloring of $G$ and $H$, define $c(u)$ as the color of some vertex $u$. Graph $G \circ H$ has vertices $(g_i, h_i)$. Every vertex in $G \circ H$ can then be assigned a 2-color combination $(c(g_i), c(h_i))$. Note that, by interpreting these 2-color combinations as simply colors, this constitutes a valid coloring of $G \circ H$ by using $\chi(G) \chi(H)$ colors. Combining inequalities \eqref{eqn:ThetaKLouwerBound} and \eqref{eqn:ThetaK_upperbound} results in
\begin{align}
    \label{eqn:MultichromBounded}
    k \omega(G) \leq \chi_k(G) = \chi(G \circ K_k) \leq k \chi(G).
\end{align}
The  above inequalities may be strict. An example is the cycle graph with five vertices and $k=2$, as $\chi_2(C_5) = 5$ \cite{stahl1976n}. Note that, by \eqref{eqn:MultichromBounded}, any upper bound on $\chi(G)$ can be transformed into an upper bound on $\chi_k(G)$.
To compute (or approximate) $\chi_k(G)$ one can consult the wide range of existing literature on standard graph coloring by using $\chi_k(G) = \chi(G \circ K_k)$. Next to that, more specific literature on multicoloring   can also be examined. \citeauthor{campelo2016lifted} \cite{campelo2016lifted}
present an integer linear programming formulation for the $k$--multicoloring of a graph and study the facial structure of the corresponding polytope. \citeauthor{malaguti2008evolutionary} \cite{malaguti2008evolutionary} use a combination of
tabu search and population management procedures as a metaheuristic to solve (slightly generalized) multicoloring problems. \citeauthor{mehrotra2007branch} \cite{mehrotra2007branch} apply branch and price to generate independent sets for solving the multicoloring problem.

\citeauthor{narasimhan1988generalization} \cite{narasimhan1988generalization} generalize $\vartheta(G)$ by introducing $\vartheta_k(G)$ as follows:
\begin{equation}
\label{eqn:ThetaK_eigenvalues}
\begin{aligned}
\vartheta_k(G) = & {\underset{A \in
\mathcal{A}(G)}{ \text{ Minimize}}}
& & \sum_{i=1}^{k}\lambda_i(A),
\end{aligned}
\end{equation}
where
\begin{equation}
    \label{eqn:AGset}
    \mathcal{A}(G) := \left \{ A \in  \mathbb{S}^n \, | \, A_{ij} = 1 \,~ \forall (i,j) \notin E(G) \right \}.
\end{equation}
\citeauthor{narasimhan1988generalization} prove that $\vartheta_k(G)$ satisfies the following inequality
\begin{equation}
\label{eqn:ThetaK_inequality}
    \alpha_k(\overline{G}) \leq \vartheta_k(\overline{G}) \leq \chi_k(G)
\end{equation}
and thus also
$\omega_k(G) \leq \vartheta_k(\overline{G})  \leq \chi_k(G)$.
Recall that $\alpha_k(G)$ is the cardinality of the largest subset $C \subseteq V$ such that the subgraph induced in $G$ by $C$, denoted $G[C]$, satisfies $\chi(G[C]) \leq k$.  Inequality \eqref{eqn:ThetaK_inequality} generalizes the Lov{\'a}sz's sandwich theorem \cite{LovaszBook}.

\citeauthor{Alizadeh} \cite{Alizadeh}
derived the following SDP formulation of $\vartheta_k(G)$, see also \cite{kuryatnikova2020maximum}:
\begin{equation}
\label{ThetaK_SDP}
\tag{$\vartheta_k$\textit{-SDP}}
\begin{aligned}
\vartheta_k(G) = & {\underset{\mu \in \mathbb{R}, \, X,Y \in
\mathbb{S}^n}{ \text{ Minimize}}}
& & \quad \langle I, Y \rangle + \mu k  \\
& \text{subject to} && X_{ij} =0\quad \forall (i,j) \notin E(G) \\
&&&  \mu I + X - J +Y \succeq 0, \ Y \succeq 0. \\
\end{aligned}
\end{equation}
The dual problem for $\vartheta_k$\textit{-SDP} is:
\begin{equation}
\label{eqn:ThetaK_SDP2}
\tag{$\vartheta_k$\textit{-SDP2}}
\begin{aligned}
\vartheta_k(G) = & {\underset{Y \in
\mathbb{S}^n}{ \text{ Maximize}}}
& & \quad \langle J, Y \rangle \\
& \text{subject to} && Y_{ij} =0\quad \forall (i,j)  \in E(G) \\
&& &  \langle I, Y \rangle = k, \ 0 \preceq Y \preceq I. \\
\end{aligned}
\end{equation}
Note that for $k=1$ constraint  $Y \preceq I$ is redundant.
We show below that $\vartheta_k(\overline{G}) \leq \chi_k(G)$. To prove the result we use different arguments than the arguments used in  \cite{narasimhan1988generalization}.
In an optimal $k$-multicoloring of $G$, define for each of the $\chi_k(G)$ colors used a vector $\vec{y}^j \in \{ 0, \, 1, \, k \}^{n+1}$, $1 \leq j \leq \chi_k(G)$. For the entries of $\vec{y}^j$, we have $\vec{y}^j_0=k$ and $\vec{y}^j_i=1$ if vertex $i$ has color $j$, 0 otherwise. Then
\begin{align*}
    \frac{1}{k^2} \sum_{j=1}^{\chi_k(G)} \vec{y}^j (\vec{y}^j)^\top = \begin{bmatrix} \chi_k(G) &  \vec{1}^\top \\
     \vec{1} & \frac{1}{k}I + \frac{1}{\chi_k(G)} X
    \end{bmatrix},
\end{align*}
for some $X \in \mathbb{S}^n$ satisfying $X_{ij} = 0$ for all $(i,j) \in E(G)$. By the Schur complement we find
$\frac{\chi_k(G)}{k}I + X - J \succeq 0.$
Simply set $Y = 0 \in \mathbb{S}^n$. Then the triple $(\frac{\chi_k(G)}{k},X,Y)$ is feasible for \ref{ThetaK_SDP} (for $\overline{G}$) with objective value $\chi_k(G)$.

To conclude this section we state the following result:
\begin{equation}
   \label{eqn:VarthetaIneq}
    \vartheta_k(\overline{G})  \leq k \vartheta(\overline{G})    \leq  \chi_k(G).
\end{equation}
\citeauthor{narasimhan1988generalization} \cite{narasimhan1988generalization} prove the first inequality in \eqref{eqn:VarthetaIneq}.
To show this, let $\widetilde{A} \in \mathcal{A}(\overline{G})$ such that $\lambda_1(\widetilde{A}) = \vartheta(\overline{G})$. Then $\vartheta_k(\overline{G}) \leq \sum \limits_{i=1} ^k \lambda_i(\widetilde{A}) \leq k \lambda_1(\widetilde{A})$ and the proof follows.
The second inequality follows from
$\vartheta(\overline{G \circ K_k}) = k\vartheta(\overline{G})$ and
$\vartheta(\overline{G \circ K_k}) \leq \chi({G \circ K_k})=\chi_k(G)$. The second inequality in \eqref{eqn:VarthetaIneq} also follows from the following known results $\vartheta(\overline{G}) \leq \chi_f(G)$  and  $k\chi_f(G) \leq \chi_k(G)$ where $\chi_f(G)$ is the fractional chromatic number of a graph, see e.g., \cite{campelo2013optimal}.
In this paper we show that  $\vartheta_k({G}) = k \vartheta({G})$ for many highly symmetric graphs.

\section{The sequence $(\vartheta_k(G))_k$} \label{sect:the_series}

In this section we consider the sequence  $\vartheta_1(G)$,  $\vartheta_2(G)$, \dots,  $\vartheta_n(G)$ where  $G$ is a graph of order $n$. We first prove that this sequence  is bounded from above (proposition \ref{Thm:ChromGeneralLB})  and increasing (proposition \ref{thm:ThetaKIncreasing}). Then, we prove that the increments of the sequence  i.e., $\vartheta_k(G) - \vartheta_{k-1}(G)$ are decreasing in $k$, see theorem \ref{thm:SecondOrderDiff}. We also show that this increment can be arbitrarily small for a particular graph, see theorem \ref{thm:DeltaKGraphName}.

Let us first establish a relation between $\vartheta_k(G)$ and $\chi(G)$.
\begin{proposition}
\label{Thm:ChromGeneralLB}
For $k \geq \chi(G)$, $G = (V, E)$, we have $\vartheta_k(G) = |V|$. Furthermore, $\vartheta_k(G) \leq \min\{ k\vartheta(G),  |V|\}$ for all $k\leq n$.
\end{proposition}
\begin{proof}
Let $k \geq \chi(G)$. Then $\alpha_k(G) = |V|$, where we  take the $k$ independent sets to be the color classes in an optimal coloring of $G$.
Thus, it follows from \eqref{eqn:ThetaK_inequality}  that $|V|\leq \vartheta_k(G) $.

Furthermore, note that for any graph $G$, matrix $J \in \mathcal{A}(G)$  is feasible for
\eqref{eqn:ThetaK_eigenvalues}.
Since matrix $J$ has eigenvalue $|V|$ with multiplicity one and the other eigenvalues equal to 0, we have
 $\vartheta_k(G) \leq |V|$ for any graph $G$.
Therefore, when $k \geq \chi(G)$ we have $\vartheta_k(G) = |V|$. Besides, $\vartheta_k(G) \leq k \vartheta(G)$ by \eqref{eqn:VarthetaIneq}.
\end{proof}
Part of proposition \ref{Thm:ChromGeneralLB} can be more succinctly stated as
$\vartheta_{\chi(G)}(G) = |V|.$
The parameter $\vartheta_k(G)$ induces a sequence of parameters for a graph, given by $\vartheta_1(G), \vartheta_2(G), \ldots, \vartheta_n(G) = |V|$. Proposition \ref{Thm:ChromGeneralLB} shows that this sequence  is bounded from above by $|V|$. The next proposition shows that this sequence is non-decreasing in $k$.

\begin{proposition}
\label{thm:ThetaKIncreasing}
For any graph $G$, $\vartheta_k(G) \leq \vartheta_{k+1}(G)$, with equality if and only if $\vartheta_k(G) =|V|$.
\end{proposition}
\begin{proof}
By proposition \ref{Thm:ChromGeneralLB}, it is enough to consider $k < n$. Consider graph $G$ of order $n$ and let $Y$ be optimal for \ref{eqn:ThetaK_SDP2}. We have $\Tr(Y) =k$ and $0 \preceq Y \preceq I$. Define matrix $Z$ as follows: $Z := \Big( 1 - \frac{1}{n-k} \Big) Y + \frac{1}{n-k} I.$ It follows that matrix $Z$ is feasible for $\vartheta_{k+1}$\textit{-SDP2} and thus
\begin{equation*}
    \vartheta_{k+1}(G) \geq \langle J, Z \rangle = \vartheta_k(G) + \frac{n - \vartheta_k(G)}{n-k} \geq \vartheta_k(G).
\end{equation*}
\end{proof}

Proposition \ref{thm:ThetaKIncreasing} allows us to further restrict \ref{ThetaK_SDP}.

\begin{proposition}
\label{prop:PositiveMu}
Let $(X^*,Y^*, \mu^*)$ be an optimal solution to \ref{ThetaK_SDP} for an arbitrary graph $G$. Then $\mu^* \geq 0$.
\end{proposition}
\begin{proof} We prove the statement by contradiction. Assume that the triple  $(X^*,Y^*, \mu^*)$ is optimal for \ref{ThetaK_SDP} and $\mu^* <0$. Note that the triple $(X^*,Y^*, \mu^*)$ is then also feasible for $\vartheta_{k+1}$\textit{-SDP}. Since $\mu^* < 0$, this would imply that $\vartheta_k(G) > \vartheta_{k+1}(G)$, which contradicts proposition \ref{thm:ThetaKIncreasing}. Thus $\mu^* \geq 0$.
\end{proof}

Next, we investigate the increments of the sequence  $(\vartheta_k(G))_k$. For that purpose, we define for any graph $G$ and $k \geq 2$ the increment of $(\vartheta_k(G))_k$ as follows:
\begin{equation}
    \label{eqn:DeltaDefinition}
    \Delta_k(G) := \vartheta_k(G) - \vartheta_{k-1}(G),
\end{equation}
and set $\Delta_1(G) = \vartheta_1(G)$.
\begin{theorem}
\label{thm:SecondOrderDiff}
For any graph $G$ and $k \geq 1$, $\Delta_k(G) \geq \Delta_{k+1}(G)$.
\end{theorem}
\begin{proof}
Let $k \geq 1$ and matrix $A_k \in \mathcal{A}(G)$, where $\mathcal{A}(G)$ is defined in \eqref{eqn:AGset}, satisfy
\begin{equation}
    \label{eqn:SecondDiff1}
    \sum_{i=1}^k \lambda_i(A_k) = \vartheta_k(G).
\end{equation}
Stated differently, matrix $A_k$ is an optimal solution to \eqref{eqn:ThetaK_eigenvalues} for computing $\vartheta_k(G)$. Since \eqref{eqn:ThetaK_eigenvalues} is a minimization problem,
\begin{equation}
\label{eqn:SecondDiff2}
\vartheta_k(G) \leq  \sum_{i=1}^k \lambda_i(A_{k'}), \, k' \neq k.
\end{equation}
By  substituting \eqref{eqn:SecondDiff1} and \eqref{eqn:SecondDiff2} in the definition of $\Delta_k(G)$ for $k \geq 2$, see \eqref{eqn:DeltaDefinition}, we obtain:
\begin{equation}
    \label{eqn:SecondDiff3}
    \Delta_k(G) \leq \sum_{i=1}^k \lambda_i(A_{k-1}) - \sum_{i=1}^{k-1} \lambda_i(A_{k-1}) = \lambda_k(A_{k-1}).
\end{equation}
Similarly,
\begin{equation}
    \label{eqn:SecondDiff4}
    \Delta_k(G) \geq \sum_{i=1}^k \lambda_i(A_k) - \sum_{i=1}^{k-1} \lambda_i(A_k) = \lambda_k(A_k).
\end{equation}
Combining \eqref{eqn:SecondDiff3} and \eqref{eqn:SecondDiff4} yields $\Delta_k(G) \geq \lambda_k(A_k) \geq \lambda_{k+1}(A_k) \geq \Delta_{k+1}(G), \, ~k \geq 2.$ The inequality $\Delta_1(G) \geq \Delta_2(G)$ follows from \eqref{eqn:VarthetaIneq}.
\end{proof}
Let us summarize the implications of propositions \ref{thm:ThetaKIncreasing} and theorem \ref{thm:SecondOrderDiff}. Proposition \ref{thm:ThetaKIncreasing} proves that
\begin{equation}
    \label{eqn:DeltaKeqZero}
    \Delta_k(G) = 0 \iff \vartheta_{k-1}(G) = |V|.
\end{equation}
For complete graphs we have $\Delta_k(K_n) = 1$, see theorem \ref{Thm:ThetaK_comple}. There exist however graphs for which $\Delta_k(G) < 1$. We investigate the limiting behaviour of $\Delta_k(G)$ in section \ref{section:Delta_k(G)}.

 When we consider the sequence  induced by $\vartheta_k(G)$ as a function of $k$, we know that this sequence  is increasing towards $|V(G)|$. Theorem \ref{thm:SecondOrderDiff} shows that the increments in this sequence decrease in $k$. Loosely speaking, one might say the second derivative of $f(k) = \vartheta_k(G)$ is negative.

\subsection{Limiting behaviour  of $\Delta_k(G)$}
\label{section:Delta_k(G)}
In this section we show that, for any real number $\varepsilon > 0$, there exists a graph $G$ and a number $k\geq 1$ such that $0 < \Delta_k(G) < \varepsilon.$ For this purpose, define graph $\GraphName = (V (\GraphName), E(\GraphName))$ as follows:
\begin{equation}
    \label{eqn:SpecialGraphDef}
    V(\GraphName) := [n] \text{ and } E(\GraphName) := \{ (i,j) \, | \, i < j \leq n-1 \} \cup \{ (n-1,n) \}.
\end{equation}
Graph $\GraphName$ is thus a complete graph on $n-1$ vertices plus one additional vertex. This additional vertex is connected to the complete graph $K_{n-1}$ by a single edge.

\begin{theorem}
\label{thm:DeltaKGraphName}
For $n \geq 5$, we have $\vartheta_{n-2}(\GraphName) = n-2+\frac{2}{n-3} \sqrt{(n-2)(n-4)}.$
\end{theorem}
\begin{proof}
We prove the theorem by finding a lower and upper bound on $\vartheta_{n-2}(\GraphName)$, both of which equal the expression stated in theorem \ref{thm:DeltaKGraphName}. Let $p = \sqrt{\frac{n-4}{(n-2)(n-3)^2}}.$ Define matrix $Y \in \mathbb{S}^n$ as follows:
\begin{equation*}
    Y = \begin{bmatrix}
    \frac{n-4}{n-3} I_{n-2} & \vec{0}_{n-2} & p \vec{1}_{n-2}   \\[1ex]
    \vec{0}_{n-2}^\top & 1 & 0 \\[1ex]
    p \vec{1}_{n-2} ^\top  & 0 & \frac{1}{n-3}
    \end{bmatrix}.
\end{equation*}
Matrix $Y$ is feasible for $\vartheta_{n-2}$\textit{-SDP2} if $0 \preceq Y \preceq I$.
Therefore we  derive
\begin{equation*}
    I-Y = \begin{bmatrix}
    \frac{1}{n-3}I_{n-2} & \vec{0}_{n-2} & -p \vec{1}_{n-2}  \\[1ex]
    \vec{0}_{n-2}^\top & 0 & 0 \\[1ex]
    - p \vec{1}_{n-2} ^\top & 0 & \frac{n-4}{n-3}
    \end{bmatrix},
\end{equation*}
and  take the Schur complement of the block $\frac{1}{n-3}I_{n-2}$ of $I-Y$:
\begin{equation*}
    \begin{bmatrix}
    0 & 0 \\[1ex]
    0 & \frac{n-4}{n-3}
    \end{bmatrix}
    - \begin{bmatrix}
    \vec{0}_{n-2}^\top \\[1ex]
    p \vec{1}_{n-2}^\top
    \end{bmatrix} (n-3)I_{n-2}
    \begin{bmatrix}
    \vec{0}_{n-2} & p \vec{1}_{n-2}
    \end{bmatrix} = \begin{bmatrix}
    0 & 0 \\ 0 & 0
    \end{bmatrix} \succeq 0.
\end{equation*}
Thus $Y \preceq I$. Similarly, by taking the Schur complement of the upper left $(n-1) \times (n-1)$ block matrix  of $Y$, we find that $Y \succeq 0$. We omit the details of this computation. This implies that $Y$ is feasible for $\vartheta_{n-2}$\textit{-SDP2} and
\begin{equation}
    \label{eqn:DeltaKproof4}
    \vartheta_{n-2}(\GraphName) \geq \langle J, Y \rangle = n-2 + \frac{2}{n-3} \sqrt{(n-2)(n-4)}.
\end{equation}
Finding the (equal) upper bound on $\vartheta_{n-2}(\GraphName)$ is a bit more involved. Let $\alpha = \frac{n-5}{n-3} \sqrt{\frac{n-2}{n-4}}$ and set
\begin{equation*}
    A = \begin{bmatrix}
    \alpha  J_{n-2}+(1-\alpha) I_{n-2} & \vec{0}_{n-2} & \vec{1}_{n-2} \\
    \vec{0}_{n-2}^\top & 1 & 0 \\
    \vec{1}_{n-2}^\top & 0 & 1
    \end{bmatrix}.
\end{equation*}
Note that $A \in \mathcal{A}(\GraphName)$, see \eqref{eqn:AGset}. We show that for
\begin{align}
    \label{eqn:DeltaKproof1}
    \beta_{1,2} &= \frac{-\alpha(n-3) \pm \sqrt{\alpha^2(n-3)^2-4(2-n)}}{2},
\end{align}
the vectors
$ v_i = [ \vec{1}_{n-2}^\top,0,\beta_i  ]^\top$, $i \in \{1, 2\}$ are two eigenvectors of matrix $A$. Consider
\begin{equation}
    \label{eqn:DeltaKproof3}
    A v_i = \begin{bmatrix}
    \big( \alpha (n-3)+1+\beta_i \big) \vec{1}_{n-2} \\[1ex]
    0 \\[1ex]
    \big( \frac{n-2}{\beta_i}+1 \big) \beta_i
    \end{bmatrix}.
\end{equation}
By \eqref{eqn:DeltaKproof1} we have that $\beta_i$, $i \in \{1, \, 2\}$ are the roots of the equation $\beta^2 + \alpha (n-3)\beta+(2-n)=0$, and so $\beta_i + \alpha(n-3) = (n-2)/ \beta_i$. Then, the right-hand side of \eqref{eqn:DeltaKproof3} equals $v_i$ scaled by the corresponding eigenvalue, which is given by the following equation:
\begin{equation}
    \label{eqn:DeltaKproof2}
    \alpha(n-3)+1+\beta_i = \frac{n-2}{\beta_i}+1.
\end{equation}
Also $u=[\vec{0}_{n-2}^\top, 1,0 ]^\top$ is an eigenvector of $A$ with corresponding eigenvalue one (and multiplicity one). Since $A$ is a real symmetric matrix, its eigenvectors are orthogonal. The remaining eigenvectors are thus $w_i=[c_i^\top,0,0]^\top$ where $c_i \in \mathbb{R}^{n-2}$ is a vector whose entries sum to 0. The eigenvectors $w_i$ correspond to eigenvalues of $1-\alpha$. We have described all eigenvectors of $A$. By substituting  \eqref{eqn:DeltaKproof1} in \eqref{eqn:DeltaKproof2} one can verify that the four unique eigenvalues of $A$ are ordered as follows:
\begin{equation*}
    \sqrt{(n-2)(n-4)} + 1 > 1 > 1 - \alpha > 1 - \sqrt{\frac{n-2}{n-4}},
\end{equation*}
with corresponding multiplicities $1$, $1$, $n-3$, $1$, respectively. The sum of the largest $n-2$ eigenvalues of $A$ serves as upper bound on $\vartheta_{n-2}(\GraphName)$, see \eqref{eqn:ThetaK_eigenvalues}. That is,
\begin{align*}
    \vartheta_{n-2}(\GraphName) &\leq \sum_{i=1}^{n-2} \lambda_i(A)= n-2 + \frac{2}{n-3} \sqrt{(n-2)(n-4)}.
\end{align*}
This upper bound on $\vartheta_{n-2}(\GraphName)$ coincides with the lower bound \eqref{eqn:DeltaKproof4}, which proves the theorem.
\end{proof}
Using theorem \ref{thm:DeltaKGraphName} we can show that $\Delta_{n-1}(\GraphName)$ $(n \geq 5)$ converges to zero. Namely,
\begin{equation*}
\Delta_{n-1}(\GraphName) = \vartheta_{n-1}(\GraphName) - \vartheta_{n-2}(\GraphName) \leq 2 \Bigg(1-\sqrt{\frac{(n-2)(n-4)}{(n-3)^2}} \Bigg),
\end{equation*}
from where it follows that  $\Delta_{n-1}(\GraphName)$ $(n \geq 5)$ converges to zero. To conclude, strictly positive values of $\Delta_k(G)$ can be arbitrarily small. It is unclear whether lower bounds exist on $\Delta_k(G)$ for fixed $k$. One example of such a bound is simple for $k=1$ i.e., $\Delta_1(G) = \vartheta_1(G) \geq \alpha(G) \geq 1.$

\section{Graph products and the generalized  $\vartheta$-number} \label{sect:Graph products}

In this section we present bounds for $\vartheta_k(G)$ when $G$ is
the strong product of two graphs (theorem \ref{Thm:StrongProdThetaK}) and the  disjunction  product of two graphs (theorem \ref{thm:GraphProdDisjunction}).

In \cite{lovasz1979shannon}, \citeauthor{lovasz1979shannon} proved the following result:
\begin{equation}  \label{eqn:LovaszThetaStrongProductEq}
\vartheta(G_1 \boxtimes G_2) = \vartheta(G_1) \vartheta(G_2),
\end{equation}
where $G_1 \boxtimes G_2$ is the strong product of $G_1$ and $G_2$, see definition \ref{Def:GraphProducts}.
Since $G_1 \boxtimes K_k$ is isomorphic to $G_1 \circ K_k$ and $\vartheta(K_k)=1$
we have that
\[
\vartheta(G \circ K_k) = \vartheta(G \boxtimes K_k) = \vartheta(G) \leq \vartheta_k(G),
\]
see also proposition \ref{thm:ThetaKIncreasing}.
Below, we generalize the result for $\vartheta(G_1 \boxtimes G_2)$ to $\vartheta_k(G_1 \boxtimes G_2)$. For that purpose we need the following well known result.
\begin{lemma}
\label{Lemma:KronEigenValues}
For square matrices $A$ and $B$ with eigenvalues $\lambda_i$ and $\mu_j$ respectively, the eigenvalues of $A \otimes B$ equal $\lambda_i \mu_j$, and so $\Tr(A \otimes B) = \Tr(A) \Tr(B)$.
\end{lemma}

For a reference of lemma \ref{Lemma:KronEigenValues} one can confer \citeauthor{horn1994topics} \cite{horn1994topics} for example.

\begin{theorem}
\label{Thm:StrongProdThetaK}
For any graphs $G_1$ and $G_2$
\begin{equation*}
    \frac{1}{k} \vartheta_k (G_1) \vartheta_k(G_2) \leq  \vartheta_k(G_1 \boxtimes G_2) \leq k \vartheta(G_1) \vartheta(G_2).
\end{equation*}
\end{theorem}
\begin{proof}
Let $X_1^*$ and $X_2^*$ be optimal to \ref{eqn:ThetaK_SDP2} for $G_1$ and $G_2$ respectively. The adjacency matrix of $G_1 \boxtimes G_2$ is given by
\begin{equation*}
    A_{G_1 \boxtimes G_2} = (A_{G_1} + I) \otimes (A_{G_2} +I) - I,
\end{equation*}
see e.g., \citeauthor{sayama2016estimation} \cite{sayama2016estimation}. Here $\otimes$ denotes the Kronecker product.
Consider $Y = \frac{1}{k} X_1^* \otimes X_2^*.$ From the adjacency matrix of $G_1 \boxtimes G_2$ it can be verified that $Y_{ij} = 0, \, \forall (i,j) \in E(G_1 \boxtimes G_2)$. By lemma \ref{Lemma:KronEigenValues}, the eigenvalues of $Y$ lie between 0 and 1 and thus $0 \preceq Y \preceq I$. More specifically, the eigenvalues of $Y$ lie between 0 and $\frac{1}{k}$, since the sum of the (non-negative) eigenvalues of $X^*_i$ is $\Tr(X^*_i)=k$ for $i=1,2$. Besides $\Tr(Y)=(1/k) \Tr(X^*_1) \Tr(X^*_2)=k$. It follows that matrix $Y$ is feasible to \ref{eqn:ThetaK_SDP2} for $G_1 \boxtimes G_2$ and attains the following objective value:
\begin{equation*}
    \langle J, Y \rangle = \frac{1}{k} \langle J, X_1^* \otimes X_2^* \rangle = \frac{1}{k} \vartheta_k(G_1) \vartheta_k(G_2).
\end{equation*}
This proves the lower bound. The upper bound follows from \eqref{eqn:VarthetaIneq} and \eqref{eqn:LovaszThetaStrongProductEq}.
\end{proof}
The bounds  from theorem \ref{Thm:StrongProdThetaK} are attained, for example, when both $G_1$ and $G_2$ are complete graphs (see theorem~\ref{Thm:ThetaK_comple}).
In general,  the bounds for $\vartheta_k(G_1 \boxtimes G_2)$ from theorem \ref{Thm:StrongProdThetaK} are more loose for larger values of $k$. We now focus on the disjunction graph product (see definition \ref{Def:GraphProducts}). For graphs $G_1$ and $G_2$ of order $n_1$ and $n_2$ respectively, we have
\begin{equation*}
    A_{G_1 \vee G_2} = J_{n_1} \otimes A_{G_2} + A_{G_1} \otimes (A_{\overline{G}_2} + I_{n_2}).
\end{equation*}
Equivalently, by noting that $A_{\overline{G}_2} + I_{n_2} = J_{n_2} - A_{G_2}$, we have $A_{G_1 \vee G_2} = \min \Big ( J_{n_1} \otimes A_{G_2} + A_{G_1} \otimes J_{n_2}, 1 \Big).$ Our next result provides an upper bound on the generalized $\vartheta$-number for  the disjunction product of two graphs.
\begin{theorem}
\label{thm:GraphProdDisjunction}
For graphs $G_1$ and $G_2$ of orders $n_1$ and $n_2$ respectively, we have
\begin{equation*}
    \vartheta_k(G_1 \vee G_2 ) \leq \min \big\{ n_1 \vartheta_k(G_2), n_2 \vartheta_k(G_1) \big\}.
\end{equation*}
\end{theorem}
\begin{proof}
Consider the SDP problem \ref{eqn:ThetaK_SDP2} for $G_1 \vee G_2$. This maximization problem is least constrained when $G_1 = \overline{K}_{n_1} $. Thus
\begin{equation}
    \label{eqn:Disconj1}
    \vartheta_k(G_1 \vee G_2 ) \leq \vartheta_k( \overline{K}_{n_1} \vee G_2).
\end{equation}
We will show that $\vartheta_k( \overline{K}_{n_1} \vee G_2) = n_1 \vartheta_k(G_2)$. Let $X^*$ be an optimal solution to \ref{eqn:ThetaK_SDP2} for $G_2$. Matrix $J_{n_1} \otimes \frac{1}{n_1}X^*$ is a feasible solution to \ref{eqn:ThetaK_SDP2} for $\overline{K}_{n_1} \vee G_2$. The objective value of this solution equals
\begin{equation}
    \label{eqn:Disconj2}
    \langle J, J_{n_1} \otimes \frac{1}{n_1} X^* \rangle = n_1 \langle J, X^* \rangle = n_1 \vartheta_k(G_2) \Longrightarrow \vartheta_k( \overline{K}_{n_1} \vee G_2 ) \geq n_1 \vartheta_k(G_2).
\end{equation}
Let ($Y^*, X^*, \mu^*$) be an optimal solution to \ref{ThetaK_SDP} for $G_2$. Then $J_{n_1} \otimes Y^*$, $J_{n_1} \otimes X^*$ and $n_1 \mu^*$ form a feasible solution to \ref{ThetaK_SDP} for $\overline{K}_{n_1} \vee G_2$. Namely, by lemma \ref{Lemma:KronEigenValues} we have that $J_{n_1} \otimes Y^* \succeq 0$.
Also
\begin{equation*}
n_1\mu^*I + J_{n_1} \otimes X^* - J + J_{n_1} \otimes Y^* =
\mu^*(n_1I_{n_1} -J_{n_1}) \otimes I_{n_2} + J_{n_1}\otimes (\mu^* I_{n_2} + X^*-J_{n_2} + Y^*)
 \succeq 0,
\end{equation*}
where we use that $\mu^*\geq 0$, see  proposition \ref{prop:PositiveMu}.
Lastly, this feasible solution to the minimization problem obtains an objective value of
\begin{equation}
    \label{eqn:Disconj3}
    \langle I, J_{n_1} \otimes Y^* \rangle + n_1\mu^*k = n_1 \Big( \langle I, Y^* \rangle + \mu^*k \Big) = n_1 \vartheta_k(G_2)~ \Longrightarrow ~\vartheta_k ( \overline{K}_{n_1} \vee G_2) \leq n_1 \vartheta_k(G_2).
\end{equation}
Now \eqref{eqn:Disconj2} and \eqref{eqn:Disconj3} imply that $\vartheta_k(\overline{K}_{n_1} \vee G_2) = n_1 \vartheta_k(G_2)$. This result combined with \eqref{eqn:Disconj1} proves that
\begin{equation}
    \label{eqn:DisjunctionHalfProof1}
    \vartheta_k(G_1 \vee G_2 ) \leq  n_1 \vartheta_k(G_2).
\end{equation}
From the definition of the disjunction graph product (see definition \ref{Def:GraphProducts}), it follows that the disjunction graph product is commutative and thus
\begin{equation}
    \label{eqn:DisjunctionHalfProof2}
    \vartheta_k(G_1 \vee G_2 ) = \vartheta_k(G_2 \vee G_1) \leq  n_2 \vartheta_k(G_1).
\end{equation}
Combining equations \eqref{eqn:DisjunctionHalfProof1} and \eqref{eqn:DisjunctionHalfProof2} proves the theorem.
\end{proof}
The proof shows that when either $G_1$ or $G_2$ is the complement of a complete graph, graph $G_1 \vee G_2$ attains the bound of theorem \ref{thm:GraphProdDisjunction}.

\section{Value of $\vartheta_k$ for some graphs}
\label{Section:ValueOfThetaKForSomeGraphs}
In \cite{lovasz1979shannon}, Lov{\'a}sz derived an explicit expression for the $\vartheta$-number
of cycle graphs and the Kneser graphs. In this section, we derive the generalized  $\vartheta$-number for those graphs, as well as  for  circulant,  complete,  complete multipartite graphs, and the Johnson graphs.
In section \ref{section:RegularGraphs} we present bounds for  $\vartheta_k(G)$ when $G$ is a regular graph and show that the bound is tight for edge-transitive graphs. Section \ref{Section:AnalysisOfCart}  provides an analysis of  $\vartheta(K_k \square G)$, which is an upper bound on  the number of vertices in the maximum $k$-colorable subgraph of $G$.

 We denote cycle graphs of order $n$ by $C_n$, complete graphs of order $n$ by $K_n$, and complete multipartite graph by $K_{m_1, \ldots, m_p}$. Note that $K_{m_1, \ldots ,m_p}$ is a graph on $n = \sum\limits_{i=1}^p m_i$ vertices.

\begin{theorem}
\label{Thm:ThetaK_comple}
For $k \leq n$, $\vartheta_k(K_n) = k$.
\end{theorem}
\begin{proof}
Consider the SDP problem \ref{eqn:ThetaK_SDP2}. For the complete graph, the only matrices feasible for \ref{eqn:ThetaK_SDP2} are diagonal matrices with trace equal to $k$. Set for example $Y = \frac{k}{n}I$. Then $Y$ is feasible for \ref{eqn:ThetaK_SDP2} and has objective value $k$.
\end{proof}

\citeauthor{stahl1976n} \cite{stahl1976n} determined $\chi_k(C_n)$. For odd cycles, he showed that $\chi_k(C_{2n+1}) = 2k+1 + \Big \lfloor \frac{k-1}{n} \Big \rfloor,$ and for even cycles that  $\chi_k(C_{2n}) = 2k$. The latter result follows trivially  from \eqref{eqn:WeaklyPerfectMulticoloring}.

Since $C_n$ is bipartite when $n$ is even, it follows from proposition \ref{Thm:ChromGeneralLB} that  $\vartheta_k(C_n) = n$ for all $k\geq 2$. To compute $\vartheta_2(C_n)$ for odd cycle graphs, we require the following lemma.
\begin{lemma}
\label{lemma:OddCycleTheta}
For $n$ odd, $n \geq 5$, we have $0.447 \approx \frac{\sqrt{5}}{5}  \leq \frac{\vartheta(C_n)}{n} < \frac{\vartheta(C_{n+2})}{n+2} < \frac{1}{2}$.
\end{lemma}
\begin{proof}

By \citeauthor{lovasz1979shannon} \cite{lovasz1979shannon}, we have
\begin{equation}
\vartheta(C_n) = \frac{n \cos( \pi /n)}{1 + \cos(\pi /n)}, \quad n \text{ odd}.
\end{equation}
Define $f(n) := \vartheta(C_n) / n$. Then
$f'(n) = \dfrac{{\pi}\sin\left(\frac{{\pi}}{n}\right)}{\left(1 + \cos\left(\frac{{\pi}}{n}\right)\right)^2n^2}.$
For $n \geq 5$, $f'(n) > 0$. Moreover, for $n \geq 5$, we have $\cos( \pi / n) < 1$. This results in $f(n) < 1/2$ and since $f(5) = \sqrt{5}/5$, this proves the lemma.
\end{proof}

Let us introduce a circulant matrix and
an edge-transitive graph. We need both terms in the proof of the following theorem. Each row of a circulant matrix equals the preceding row in the matrix rotated one element to the right. Circulant matrices thus have a constant row sum. This constant row sum is also one of the eigenvalues with $\vec{1}$ as its corresponding eigenvector.
A graph is edge transitive if its automorphism group acts transitively on edges, i.e., if for every two edges there is an automorphism that maps one to the other.

\begin{theorem}
\label{Thm:OddCycleTheta2}
Let $n$ be odd and $n>1$. Then $\vartheta_2(C_n) = 2 \vartheta(C_n)$  and $\vartheta_k(C_n) =n$ for all $k\geq 3$. 
\end{theorem}
\begin{proof}
For $n = 3$, $C_3 = K_3$ and the result follows from theorem \ref{Thm:ThetaK_comple}. Thus let $n \geq 5$. Let $\Gamma \subset \mathbb{S}^n$ be the set of optimal feasible solutions to $\vartheta_1$\textit{-SDP2} for $C_n$ and let $Y \in \Gamma$. Note $\Gamma$ is convex. Let $p(Y)$ denote an optimal solution to $\vartheta_1$\textit{-SDP2} obtained by permuting the vertices of $C_n$ by automorphism $p$. Matrix $p(Y) \in \Gamma$. Denote the average over all automorphisms $p$ by $\bar{Y}$. Then $\bar{Y} \in \Gamma$ by convexity of $\Gamma$ and since $C_n$ is edge transitive, $\bar{Y}$ is a circulant matrix, like the adjacency matrix of $C_n$.

As $\bar{Y} \in \Gamma$, we find
\begin{equation}
\label{eqn:ProofOddCycle}
\langle J, \bar{Y} \rangle = \Tr (\vec{1} \vec{1}^\top \bar{Y})  =  \Tr (\vec{1}^\top \bar{Y} \vec{1} )= \vec{1}^\top \bar{Y} \vec{1}= \vartheta_1(C_n).
\end{equation}
As $\bar{Y}$ is also circulant, it has eigenvector $\vec{1}$. By \eqref{eqn:ProofOddCycle}, its corresponding eigenvalue equals $\bar{\lambda} = \vartheta(C_n) /n$.

We will prove that the largest eigenvalue of $\bar{Y}$ equals $\bar{\lambda}$. Assume that the largest eigenvalue of $\bar{Y}$ does not equal $\bar{\lambda}$. Then $\bar{Y}$ has eigenvalue $\Lambda$, for some $\Lambda > \bar{\lambda}$. Since $\bar{Y}$ is a symmetric circulant matrix of odd dimension, $\bar{Y}$ has only one eigenvalue with odd multiplicity (\citeauthor{tee2007eigenvectors} \cite{tee2007eigenvectors}). Thus $\Lambda$ or $\bar{\lambda}$ have multiplicity greater than one. Note that since $\bar{Y}$ is feasible for $\vartheta_1$\textit{-SDP2}, it has non-negative eigenvalues that sum to one. However, both terms $\Lambda + 2\bar{\lambda}$ and $2 \Lambda + \bar{\lambda}$ are strictly greater than one by lemma \ref{lemma:OddCycleTheta}, and hence, the assumption that $\bar{\lambda}$ is not the largest eigenvalue of $\bar{Y}$ leads to a contradiction.

The largest eigenvalue of $\bar{Y}$ is thus smaller than 1/2. Then $2 \bar{Y} \preceq I$. Clearly, $2\bar{Y}$ satisfies the other feasibility conditions of $\vartheta_2$\textit{-SDP2}. Thus $2\bar{Y}$ is feasible for $\vartheta_2$\textit{-SDP2} and $\vartheta_2(C_n) \geq 2 \vartheta(C_n)$. Combined with \eqref{eqn:VarthetaIneq}, the theorem follows.

Since $\chi(C_n) = 3$ for odd cycles, $\vartheta_3(C_n) = n$ follows trivially from proposition \ref{Thm:ChromGeneralLB}.
\end{proof}

Graphs for which the adjacency matrix is a circulant matrix are called \textit{circulant graphs}, like the cycle graphs and some Paley graphs. There has been research done on computing $\vartheta(G)$ for circulant graphs \cite{bachoc2013theta, brimkov2007algorithmic, brimkov2000lovasz, crespi2004exact}.
In particular, \citeauthor{crespi2004exact} \cite{crespi2004exact} computes the Lov\'asz theta function for the circulant graphs of degree four having even displacement, while
 \citeauthor{brimkov2000lovasz} \cite{brimkov2000lovasz} consider $\vartheta(C_{n,j})$, where $V(C_{n,j}) = \{0, 1, \ldots ,n-1 \}$ and $E(C_{n,j}) = E(C_n) \cup \{ (i,i')\, | \, i-i'= j \mod n \}$.

Let $H_n$ be a connected circulant graph on $n$ vertices. Then $H_n$ contains a Hamiltonian cycle (\citeauthor{boesch1984circulants} \cite{boesch1984circulants}). Equivalently, the cycle graph $C_n$ is a minor of $H_n$. Maximization problem \ref{eqn:ThetaK_SDP2} is then more restricted for $H_n$ then it is for $C_n$. Thus
\begin{equation*}
    \vartheta_1(H_n) \leq \vartheta_1(C_n) \leq \frac{n}{2}.
\end{equation*}
Consider $\vartheta_1$\textit{-SDP2} for $H_n$. Graph $H_n$ has a circulant adjacency matrix, meaning we can restrict optimization of  $\vartheta_1$\textit{-SDP2} over  the Lee scheme, the association scheme of symmetric circulant matrices, without loss of generality \cite{Parillo}. As \eqref{eqn:ProofOddCycle} shows, $\vartheta_1$\textit{-SDP2} is now equivalent to maximizing the largest (scaled) eigenvalue over feasible matrices. Let $M$ be a matrix optimal for $\vartheta_1$\textit{-SDP2} for graph $H_n$. Then $\lambda_1(M) = \vartheta(H_n)/n \leq 1/2$. Then $2M$ is also optimal for $\vartheta_2$\textit{-SDP2} for  graph $H_n$. More generally, if $k \leq n/\vartheta(H_n)$, then $\lambda_1(kM) \leq 1$ and $kM$ is then feasible for \ref{eqn:ThetaK_SDP2}, attaining the objective value $\min \{ k \vartheta(H_n), n \}$.
In case $k > n / \vartheta(H_n)$, we  have $\vartheta_k(H_n) = n$.
Thus, in general
\begin{equation}
\label{eqn:ConstantRowSumResult}
\vartheta_k(H_n) = \min \{ k \vartheta(H_n), n \}.
\end{equation}
For any $k$, there exists a circulant graph $P$ on $n$ vertices such that $\vartheta_k(P) < n$. Specifically, if $P$ is the Paley graph of order $n$, then $\vartheta(P) = \sqrt n$ (cf. \cite{gvozdenovic2008approximating}). For fixed $k$ and $n$ large enough, $k \sqrt n < n$.

\begin{theorem}
\label{thm:CompleteMultipartite}
For $m_1 \geq m_2 \geq \ldots \geq m_p$ and $k \leq p$, $\vartheta_k(K_{m_1, \ldots, m_p}) = \sum \limits_{i=1}^k m_i$.
\end{theorem}
\begin{proof} Let $n = \sum_{i=1}^p m_i$.
For notational convenience, we  write $K = K_{m_1, \ldots, m_p}$, with corresponding adjacency matrix $A_K$.  Since $K$ is a graph on $n$ vertices, $A_K$ can be written as $A_K = J_{n} - \text{Diag}(J_{m_1}, \ldots, J_{m_p})$. Note that  $X := \text{Diag}(J_{m_1}, \ldots, J_{m_p}) \in \mathcal{A}(K)$, see  $\eqref{eqn:AGset}$. Therefore $X$ is feasible for \eqref{eqn:ThetaK_eigenvalues}. The eigenvalues of $X$ are the eigenvalues of the block matrices $J$. Then, $\lambda_i(X) = m_i$ for $i \in [p]$.  Thus, we have $\vartheta_k(K) \leq \sum_{i=1}^k  \lambda_i(X) = \sum_{i=1}^k m_i.$ Note that $\alpha_k(K) = \sum_{i=1}^k m_i$, and the proof follows from \eqref{eqn:ThetaK_inequality} and the above inequality.
\end{proof}
Recall again the definition of $\Delta_k(G)$, given in \eqref{eqn:DeltaDefinition}. In section \ref{section:Delta_k(G)}, we show  that strictly positive values $\Delta_k(G)$ can be arbitrarily small. We show now, by use of theorem \ref{thm:CompleteMultipartite}, that the ratio between strictly positive successive values of $\Delta_k(G)$ can be arbitrarily small. More formally, for any $\varepsilon > 0$ and any $k \geq 1$, there exists a graph $G$ such that
\begin{equation}
    \label{eqn:DeltaKRatioIneq}
    0 < \frac{\Delta_{k+1}(G)}{\Delta_k(G)} < \varepsilon.
\end{equation}
We again ignore the case $\Delta_k(G) = 0$, see \eqref{eqn:DeltaKeqZero}. In view of  theorem \ref{thm:CompleteMultipartite}  we have
\begin{equation*}
    \frac{\Delta_{2}(K_{n,1})}{\Delta_1(K_{n,1})} = \frac{1}{n} < \varepsilon,
\end{equation*}
 for some integer $n$ sufficiently large.
Thus for sufficiently large $n$, graph $K_{n,1}$ satisfies \eqref{eqn:DeltaKRatioIneq} for $k=1$. Graph $K_{n,n,1}$ satisfies \eqref{eqn:DeltaKRatioIneq} for $k = 2$. Graph $K_{n,n,n,1}$ satisfies \eqref{eqn:DeltaKRatioIneq} for $k = 3$, and so on.\\

\subsection{Regular graphs }
\label{section:RegularGraphs}
In this section we present an upper bound on  the $\vartheta_k$-function for regular graphs, see theorem \ref{Thm:EdgeTransitiveThetaK}. This result can be seen as a generalization of the Lov\'asz upper bound on the  $\vartheta$-function for regular graphs. We exploit the result of theorem  \ref{Thm:EdgeTransitiveThetaK} to  derive an explicit expression for the generalized theta function for the Kneser graph, see theorem \ref{Thm:KneserGraphThetak}. Moreover, we prove that $\vartheta_k(G)=k\vartheta (G)$ when $G$ is the Johnson graph, see theorem \ref{Thm:JhonsonGraphThetak}.

Let us first state the following well known result.
\begin{theorem}[\citeauthor{lovasz1979shannon} \cite{lovasz1979shannon}]
\label{lemma:LovaszThetaUB}
For a regular graph $G$ of order $n$, having adjacency matrix $A_G$ and $\lambda_1(A_G) \geq \lambda_2(A_G) \geq \ldots \geq \lambda_n(A_G)$, we have
\begin{equation*}
    \vartheta(G) \leq \frac{n \lambda_n(A_G)}{\lambda_n(A_G) - \lambda_1(A_G)}.
\end{equation*}
If $G$ is an edge-transitive graph, this inequality holds with equality.
\end{theorem}
For a finite set of real numbers $P$, we denote by $ \vec{S}_k(P)$ the sum of the largest $k$ elements in  $P$. Now, we  state our result.
\begin{theorem}
\label{Thm:EdgeTransitiveThetaK}
For any regular graph $G$ of order $n$, we have
\begin{equation}
    \vartheta_k(G) \leq \min_x \vec{S}_k(\sigma(J+xA_G)) \leq n+\frac{n}{\lambda_n(A_G) - \lambda_1(A_G)} \big( \lambda_1(A_G) + \sum_{i=0}^{k-2} \lambda_{n-i}(A_G) \big),
\end{equation}
where we set the summation equal to 0 when $k=1$ and $\sigma( \cdot )$ denotes the spectrum of a matrix. The first inequality holds with equality if $G$ is also edge-transitive.
\end{theorem}
\begin{proof}
The proof is an extension of \citeauthor{lovasz1979shannon}' \cite{lovasz1979shannon} proof of theorem \ref{lemma:LovaszThetaUB}. Let $G$ be a regular graph of order $n$. For notational convenience, we write $A_G = A$ and $\lambda_i(A_G) = \lambda_i$. Since $G$ is a regular graph, vector $\vec{1}$ is an eigenvector of $A$. Let $v \neq \vec{1}$ be an eigenvector of $A$. As $A$ is symmetric, its eigenvectors are orthogonal. Thus $\vec{1}^\top v = 0$, which implies that $J v = 0$. Thus the eigenvectors of $A$ are also eigenvectors of $J+xA$. In particular, we have
\begin{align}
\label{eqn:matrixJ_spectrumTransform}
\sigma(J+xA) = \{ n + x \lambda_1, \, x \lambda_2, \,  \ldots , \, x \lambda_n \},
\end{align}
for any $x \in \mathbb{R}$. Note that $J+xA \in \mathcal{A}(G)$, see \eqref{eqn:AGset}. Therefore, it follows from  \eqref{eqn:ThetaK_eigenvalues} that
\begin{equation*}
    \vartheta_k(G) \leq \underset{x}{\min} \, \vec{S}_k(\sigma(J+x A)).
\end{equation*}
Minimizing $\vec{S}_k(\sigma(J+xA))$ can be done analytically when $k=1$. \citeauthor{lovasz1979shannon} \cite{lovasz1979shannon} showed that setting $x = \frac{n}{\lambda_n - \lambda_1}$ minimizes $\vec{S}_k(\cdot)$ when $k=1$. Setting $x$ to this negative value provides the second upper bound in the theorem. Note that $x = \frac{n}{\lambda_n - \lambda_1}$ implies that $n + x\lambda_1 = x\lambda_n$.

We now prove that that $\vartheta_k(G) = \min_x \vec{S}_k(\sigma(J+xA))$ when $G$ is edge-transitive. Assume that $G$ is edge-transitive. It is known that the sum of the $k$ largest eigenvalues of a matrix is a convex function (\citeauthor{overton1993optimality} \cite{overton1993optimality}). Thus the average over all optimal solutions to \eqref{eqn:ThetaK_eigenvalues} of all automorphisms of $G$ is also optimal. Since $G$ is edge-transitive, this average is of the form $J+xA$, which proves the equality claim.
\end{proof}

We remark that theorem \ref{Thm:OddCycleTheta2} can also be proven by applying theorem \ref{Thm:EdgeTransitiveThetaK}.

To obtain sharper bounds for $\vartheta_k(G)$, one can minimize $\vec{S}_k(\sigma(J+xA))$, or compute $\vartheta_k(G)$ directly. Note that computing $\vartheta_k(G)$ by interior point methods is computationally demanding  already for some graphs with 200 vertices, see \cite{kuryatnikova2020maximum}.
In general, $\vec{S}_k(\sigma(J+xA))$ is the sum of the $k$ largest linear functions given by $\sigma(J+xA)$. \citeauthor{ogryczak2003minimizing} \cite{ogryczak2003minimizing} consider this problem which they show is solvable in linear time, but unfortunately, obtaining a general solution is not possible.
When we consider specific graphs, and $\sigma(J+xA)$ is thus explicit, minimizing $\vec{S}_k(\cdot)$ can be done analytically, as we show for the Kneser graph.\\

\citeauthor{lovasz1979shannon} \cite{lovasz1979shannon} proved that
$ \vartheta(K(n,m)) = \binom{n-1}{m-1},$
where $K(n,m)$ is the Kneser graph, see definition \ref{Def:KneserGraph}.
The Kneser graph $K(n,m)$ is regular of valency $ \binom{n-m}{m}$.
We provide an explicit expression for $\vartheta_k(K(n,m))$.
\begin{theorem}
\label{Thm:KneserGraphThetak}
For $k \leq \lfloor \frac{n}{m} \rfloor$ and $1 \leq k \leq n-2m+1$, we have
\begin{equation*}
    \vartheta_k(K(n,m)) = k \vartheta(K(n,m)) = k \binom{n-1}{m-1}.
\end{equation*}
When $k > \frac{n}{m}$ or $k > n - 2m +1$
\begin{equation*}
    \vartheta_k(K(n,m)) = \binom{n}{m}.
\end{equation*}
\end{theorem}
\begin{proof}
Note that $n$ does not refer to the number of vertices but to a parameter of the Kneser graph $K(n,m)$. We will use $v$ to denote the number of vertices of $K(n,m)$, i.e.,  $v = \binom{n}{m}$. Let $A$ be the adjacency matrix of $K(n,m)$, having eigenvalues $\lambda_1 \geq \ldots \geq \lambda_v$.
We compute the minimum of $\vec{S}_k(\sigma(J+xA))$, see theorem \ref{Thm:EdgeTransitiveThetaK}. Recall that  $\sigma(\cdot)$ denotes the spectrum of a matrix, and  $\vec{S}_k(\sigma(\cdot))$ the sum of $k$ largest eigenvalues in a matrix. Define $f_{k}(x) := \vec{S}_k(\sigma(J+xA)).$ For  $x^* = \frac{v}{\lambda_v - \lambda_1} <0$ and $k \leq v$, we have
\begin{equation}
    \label{eqn:KneserThetaK_eq1}
    f_k(x^*) = v + x^*( \lambda_1 + \sum_{i=0}^{k-2} \lambda_{v-i} ).
\end{equation}
The greatest and smallest eigenvalue of $A$ equal $\lambda_1 = \binom{n-m}{m}$ and $\lambda_v = -\binom{n-m-1}{m-1}$, with corresponding multiplicities 1 and $n-1$, see \cite{lovasz1979shannon}. Thus, in the case that $k \leq \lfloor n/m  \rfloor \leq n$, function $f_k$ is determined only by $\lambda_1$ and $\lambda_v$. More precisely,
\begin{equation}
    \label{eqn:KneserThetaK_eq2}
    f_k(x^*) = v + x^* \Bigg[ \binom{n-m}{m} - (k-1) \binom{n-m-1}{m-1} \Bigg] = v+x^*(\lambda_1 + (k-1)\lambda_v).
\end{equation}
Since $\vartheta(K(n,m)) = \vec{S}_1(\sigma(J+x^*A)) = x^* \lambda_v = v + x^* \lambda_1$ (see \cite{lovasz1979shannon}), we can rewrite \eqref{eqn:KneserThetaK_eq2} as
\begin{align}
    \label{eqn:KneserThetaK_eqRev1}
    f_k(x^*) = k x^* \lambda_v = k \vartheta(K(n,m)).
\end{align}
We show now that $x^*$ minimizes $f_k$ when $k< \frac{n}{m}$, $1 \leq k \leq n-2m+1$.  For any $\varepsilon >0$
we have
\begin{equation*}
    f_k(x^* + \varepsilon) = f_k(x^*) + \varepsilon \Bigg[ \binom{n-m}{m} - (k-1) \binom{n-m-1}{m-1} \Bigg],
\end{equation*}
and thus
$k < \frac{n}{m} \Longrightarrow \binom{n-m}{m} - (k-1) \binom{n-m-1}{m-1} > 0 ~\Longrightarrow~ f_k(x^* + \varepsilon) > f_k(x^*).$ Similarly,  for any sufficiently small $\varepsilon > 0$, we have
\begin{equation} \label{eqn:KneserThetaK_eq3}
 f_k(x^* - \varepsilon) = k(x^* - \varepsilon) \lambda_v = kx^* \lambda_v - k\varepsilon \lambda_v >x^* \lambda_v + (k-1) x^* \lambda_v = f_k(x^*).
\end{equation}
Here we used the fact that $x^* \lambda_v = v+x^* \lambda_1$. From the previous discussion it follows that for any sufficiently small positive $\varepsilon$ we have
\begin{equation*}
    f_k(x^*) < f_k(x^* \pm \varepsilon).
\end{equation*}
By convexity of $f_k$, $x^*$ is the global minimizer of $f_k$. From~\eqref{eqn:KneserThetaK_eqRev1}, the theorem follows for the case $k < \frac{n}{m}$.

Now we consider the case $k = \frac{n}{m}$, $1 \leq k \leq n-2m+1$. From \eqref{eqn:KneserThetaK_eq2} it follows that $f_k(x^*) = v.$ In fact, for any $\beta$ satisfying $x^* \leq \beta \leq 0$, $f_k(\beta) = v$. For any $\varepsilon > 0$,
\begin{equation}
    f_k(\varepsilon) = v + \varepsilon \Big( \sum_{i=1}^k \lambda_i \Big).
\end{equation}
As the $\lambda_i$ sum to zero, the sum of the $k$ largest $\lambda_i$ must be strictly positive. Thus $f_k(\varepsilon)>v$. The derivation from \eqref{eqn:KneserThetaK_eq3} is also valid for the case $k = \frac{n}{m}$. Invoking again the convexity of $f_k$ proves that $v$ is the minimum value of $f_k$. Thus
\begin{equation*}
\vartheta_{n/m}(K(n,m)) = \frac{n}{m} \binom{n-1}{m-1} = \binom{n}{m} = v.
\end{equation*}
It follows from proposition \ref{thm:ThetaKIncreasing} that $\vartheta_k(K(n,m)) = v$ for $k > n/m$. Lastly, Kneser's conjecture (\citeauthor{kneser1955aufgabe}~\cite{kneser1955aufgabe}), which was proved by \citeauthor{lovasz1978kneser} \cite{lovasz1978kneser}, states that $\chi(K(n,m)) = n-2m+2$. The inequality $k > n - 2m+1$ is thus equivalent to $k \geq \chi(K(n,m))$. Therefore, we can apply proposition \ref{Thm:ChromGeneralLB} to prove the last claim.
\end{proof}

Since the Johnson graphs  (definition \ref{Def:JohnsonGraphs}) are edge-transitive (see e.g., \citeauthor{chen1987hamiltonian}  \cite{chen1987hamiltonian}) we can apply theorem \ref{Thm:EdgeTransitiveThetaK} to compute  the corresponding generalized $\vartheta$-number. The next theorem  generalizes theorem \ref{Thm:KneserGraphThetak}.
\begin{theorem}  \label{Thm:JhonsonGraphThetak}
For $0 \leq f < m$, $k \leq n$ and the Johnson graph $J(n,m,  f )$, it follows
\begin{equation*}
    \vartheta_k(J(n,m,   f  )) = \min \Bigg\{ k \vartheta \left (J(n,m, f  \right ) ), \binom{n}{m} \Bigg\}.
\end{equation*}
\end{theorem}
\begin{proof}
Let  $v$ denote the order of $J(n,m, f  )$, i.e., $v = \binom{n}{m}$. The multiplicities $\mu_i$ of the (not necessarily distinct) $m+1$ eigenvalues $\lambda_i$ of the adjacency matrix $A$ of $J(n,m,f )$ are
\begin{equation*}
    \mu_i = \binom{n}{i} - \binom{n}{i-1}, \quad 0 \leq i \leq m,
\end{equation*}
see e.g., \citeauthor{brouwer1989special} \cite{brouwer1989special}.
We set $\mu_0 = 1$, corresponding to the multiplicity of $\lambda_1$. The multiplicities are unordered, that is, the multiplicity of $\lambda_i$ does not necessarily equal $\mu_i$. When $n \leq 4$, the theorem can be verified numerically. We will now assume that $n>4$. Then $\mu_i \geq n-1$, and in particular, the multiplicity of $\lambda_v$ is at least $n-1$. From theorem \ref{Thm:EdgeTransitiveThetaK} it follows that
\begin{equation*}
    \vartheta_k(J(n,m,   f )) = \min \vec{S}_k(\sigma(J+xA)).
\end{equation*}
Define
$f(x) := \vec{S}_k(\sigma(J+xA)).$ Because the multiplicity of $\lambda_v$ is at least $n-1$, for $k \leq n$, we have that $x^* = n/(\lambda_v - \lambda_1)$ minimizes $f(x)$, and $f(x^*) = v +x^*(\lambda_1 + (k-1) \lambda_v).$ If $f(x^*) > n$, the minimum will occur at $f(0) = n$.
\end{proof}
We can explicitly compute $\vartheta_k(J(n,m,  m-1))$. Taking the eigenvalues of this graph from  \cite{brouwer1989special}, chapter 9, we find
\begin{equation*}
    \vartheta_k(J(n,m,   m-1)) =  \frac{k}{n+1}\binom{n+1}{m}, \quad k \leq n-m+1.
\end{equation*}
For $k^*=n-m+1$, we see that $\vartheta_{k^*}(J(n,m,  m-1))$ equals the number of vertices in $J(n,m, m-1 )$.

\subsection{Relation between $\vartheta(K_k \square G)$ and $\vartheta_k(G)$}
\label{Section:AnalysisOfCart}

\citeauthor{gvozdenovic2008operator} \cite{gvozdenovic2008operator} show how to exploit an upper bound on the independence number of a graph to obtain a lower bound for the chromatic number of its complement graph.
They do not consider  the generalized $\vartheta$-number  in the bounding procedure.
\citeauthor{kuryatnikova2020maximum} \cite{kuryatnikova2020maximum} exploit the generalized $\vartheta$-number to compute bounds on the chromatic number of a graph.
For some graphs,  the lower bounds on $\chi(G)$ from \cite{kuryatnikova2020maximum}  coincide with the bounds obtained by using the theta function  as suggested by  \cite{gvozdenovic2008operator}.
Here we explain that finding by analyzing $\vartheta(K_k \square G)$ for symmetric graphs.
We also show that the gap between  $\vartheta_k(G)$ and  $\vartheta(K_k \square G)$ can be  arbitrarily large.

\citeauthor{chvatal1973edmonds} \cite{chvatal1973edmonds}  noted that
\begin{equation*}
    \alpha_k(G) = |V(G)| \iff \chi(G) \leq k.
\end{equation*}
Stated differently, $\chi(G) = \min \{ k \, | \,  k\in \mathbb{N}, \, \alpha_k(G) = |V(G)|  \,\},$ or in plain words,  the $k$ independent sets giving $\alpha_k(G)$ correspond to the color classes of $G$ in an optimal coloring.
Analogue to $\chi_k(G) = \chi(G \circ K_k)$, it is known (cf. \cite{kuryatnikova2020maximum}) that
\begin{equation}
    \label{eqn:MKCStoMS}
    \alpha_k(G) = \alpha(K_k \square G),
\end{equation}
where $K_k  \square G$ is the graph Cartesian product, see definition \ref{Def:GraphProducts}.
For a graph parameter $\beta( G )$ that satisfies
\begin{equation*}
  \alpha(G) \leq \beta(G) \leq \chi(\overline{G}),
\end{equation*}
\citeauthor{gvozdenovic2008operator} \cite{gvozdenovic2008operator} define $\mathrm{\Psi}_\beta(G)$ as follows:
\begin{equation*}
    \mathrm{\Psi}_\beta(G) := \min \{ k \, | \, k\in \mathbb{N}, \,  \beta(K_k \square G) = |V(G)| \,   \}.
\end{equation*}
Then $\mathrm{\Psi}_\alpha (G) = \chi(G)$. The operator $\mathrm{\Psi}_\beta(\cdot)$ can be applied to a variety of graph parameters $\beta(G)$ and enables obtaining a hierarchy  of bounds for $\chi(G)$ from a hierarchy  of bounds for $\alpha(G)$. For example, when $\beta(G) = \vartheta(G)$  \citeauthor{gvozdenovic2008operator} \cite{gvozdenovic2008operator} show that $\mathrm{\Psi}_\vartheta (G) = \lceil \vartheta( \overline{G} ) \rceil.$ It follows from  \eqref{eqn:MKCStoMS} that parameters $\vartheta_k(G)$ and $\vartheta(K_k \square G)$ both provide upper bounds on $\alpha_k(G)$. Therefore, it is  natural to compare $\mathrm{\Psi}_\vartheta (G)$ with
\begin{equation}
    \label{eqn:PsiThetaCartEquality}
    \mathrm{\Psi}_{\vartheta_k} (G) = \min \{ k \, | \, k\in \mathbb{N}, \, \vartheta_k(G) = |V(G)| \, \}.
\end{equation}
This  comparison boils down to the comparison of $\vartheta(K_k \square G)$ and  $\vartheta_k(G)$.
Numerical results in  \cite{kuryatnikova2020maximum}  suggest the following conjecture.
\begin{conjecture}
\label{conj:SDPinequalities}
For any graph $G$ and any natural number $k$,
$\vartheta(K_k \square G) \leq \vartheta_k(G).$ Equality holds when $\vartheta_k(G) = k \vartheta(G)$.
\end{conjecture}
We show below that the gap between $\vartheta_k(G)$ and $\vartheta(K_k \square G)$ can be made arbitrarily large.
We first state the following lemma that is needed in the rest of this section.
\begin{lemma}[\citeauthor{gvozdenovic2008operator} \cite{gvozdenovic2008operator}]
\label{Lemma:BlockFormSDP}
Given $A,B \in \mathbb{S}^n$ and $Y = I_k \otimes A + (J_k - I_k) \otimes B$, then $ Y \succeq 0$ if and only if $A - B \succeq 0$ and $A + (k-1)B \succeq 0$. Furthermore, $\sigma(Y) = \sigma(A+(k-1)B) \cup \sigma(A-B)^ {\{k-1\} }.$
\end{lemma}
Now, we are ready to present our result.
\begin{proposition} \label{prop:Large gap}
For any number $M \geq 0$, there exists a graph $G$ and integer $k$ such that
\begin{equation*}
    \vartheta_k(G) - \vartheta(K_k \square G) \geq M.
\end{equation*}
\end{proposition}

\begin{proof}
Consider again graph $\GraphName$, as defined in \eqref{eqn:SpecialGraphDef} for even $n$ and set $k = n/2$. We will show that $\vartheta_{n/2}(\GraphName) - \vartheta(K_{n/2} \square \GraphName)$ is increasing in $n$. Let $p = 1/(2\sqrt{n-2})$ and consider first
\begin{equation*}
    X = \begin{bmatrix}
    \frac{1}{2} I_{n-2} & \vec{0}_{n-2} & p\vec{1}_{n-2} \\[1ex]
    \vec{0}^\top_{n-2} & \frac{1}{2} & 0 \\[1ex]
    p \vec{1}^\top_{n-2} & 0 & \frac{1}{2}
    \end{bmatrix}.
\end{equation*}
Taking the Schur complement of the bottom right $2 \times 2$ block of $X$ shows that $0 \preceq X \preceq I$ (see the proof of theorem \ref{thm:DeltaKGraphName} for more details). Combined with the fact that $\langle I, X \rangle = k$, it follows that $X$ is feasible for \ref{eqn:ThetaK_SDP2}. Hence,
\begin{equation}
    \label{eqn:ThetaHalfV1}
    \vartheta_{n/2}(\GraphName) \geq \langle J, X \rangle = n/2 + \sqrt{n-2}.
\end{equation}
As for $\vartheta(K_{n/2} \square G)$, let
\begin{equation*}
    A = \begin{bmatrix}
    -kJ_{n-1} + (k+1)I_{n-1} & \vec{1}_{n-1} \\[1ex]
    \vec{1}^\top_{n-1} & 1
    \end{bmatrix}, \,
    ~B = \begin{bmatrix}
    J_{n-1} & \vec{1}_{n-1} \\[1ex]
    \vec{1}^\top_{n-1} & -k
    \end{bmatrix},
\end{equation*}
and set $Y := I \otimes A + (J-I) \otimes B$.
Then matrix $Y \in \mathcal{A}(K_{n/2} \square G)$, see \eqref{eqn:AGset}. Furthermore, matrix $Y$ is of the form described in lemma \ref{Lemma:BlockFormSDP}. Then the largest eigenvalue of $Y$ satisfies
$\lambda_1(Y) = \max\{ \lambda_1(A-B), \, \lambda_1(A+(k-1)B) \}.$ Similar to the methods used in the proof of theorem \ref{thm:DeltaKGraphName}, it can be shown that $\lambda_1(Y) = k+1$. Thus,
\begin{equation}
    \label{eqn:ThetaHalfV2}
    \vartheta (K_{n/2} \square \GraphName) \leq \lambda_1(Y) = n/2+1.
\end{equation}
Combining \eqref{eqn:ThetaHalfV1} and \eqref{eqn:ThetaHalfV2} for fixed $M$ and large enough (even) $n$, gives $\vartheta_{n/2}(\GraphName) - \vartheta(K_{n/2} \square \GraphName) \geq \sqrt{n-2} - 1 \geq M.$

\end{proof}
We prove conjecture \ref{conj:SDPinequalities} only for a particular  class of graphs. Let us first show the following result.
\begin{theorem}
\label{thm:CartesianThetaK}
Let $G$ be graph of order $n$ that is both edge-transitive and vertex-transitive. Then
\begin{equation*}
    \vartheta(K_k \square G) =  \min \{ k \vartheta(G) , n \} = \vartheta_k(G).
\end{equation*}
\end{theorem}
\begin{proof}
For notational convenience we denote $A=A_G$. Because $G$ is regular, edge-transitive and vertex-transitive, we may assume without loss of generality that $\mathcal{A}(K_k \square G)$, see \eqref{eqn:AGset}, contains only matrices of the form $X = I_k \otimes (J_n+ x A) + (J_k - I_k) \otimes (J_n+yI_n).$ In order to minimize the largest eigenvalue of $X$ we apply lemma \ref{Lemma:BlockFormSDP} and find
\begin{align*}
    \lambda_1(X) = f(x,y) &= \max \Big\{ \lambda_1(xA - yI_n), \lambda_1(kJ_n+xA+(k-1)yI_n) \Big\} \\
    &= \max \begin{cases}
    f_1 = x \lambda_1 -y, \\
    f_2 = x \lambda_n - y, \\
    f_3 = kn+x\lambda_1+(k-1)y, \\
    f_4 = x \lambda_n+(k-1)y,
    \end{cases}
\end{align*}
where $\lambda_1$ and $\lambda_n$ are the greatest and smallest eigenvalue of $A$ respectively. We have used that $\sigma(k J_n + xA + (k-1)y I_n)$ can be expressed similarly to \eqref{eqn:matrixJ_spectrumTransform}. We minimize $\lambda_1(X)$ by considering different intervals of $x$. In case $x \geq 0$, we have $f(x,y) = \max\{ f_1, \, f_3 \}$, which is minimized when $x=0$ and $f_1 = f_3$. Solving $f_1 = f_3$ for $y$, when $x=0$, yields $y = -n$. Thus, when $x \geq 0$, we find that $f(0,-n) = n$ is the minimum. Furthermore,
\begin{equation*}
    \frac{kn}{\lambda_n - \lambda_1} \leq x \leq 0 \Longrightarrow f(x,y) = \max \{ f_2, f_3 \}.
\end{equation*}
The minimum here is attained when
$$
 f_2 = f_3 ~\Longrightarrow~ y = x \Bigg( \frac{\lambda_n - \lambda_1}{k} \Bigg) - n
    ~\Longrightarrow~ f_2 = n+ \frac{1}{k}\left( (k-1)\lambda_n+\lambda_1\right )x.
$$
Depending on the sign of $(k-1)\lambda_n + \lambda_1$ we find either $f(0,-n) = n$ or $f(\frac{kn}{\lambda_n-\lambda_1},0) = k \frac{n \lambda_n}{\lambda_n - \lambda_1} = k \vartheta(G)$, by theorem \ref{lemma:LovaszThetaUB}.
Lastly, consider the case $x{ \leq} \frac{kn}{\lambda_n - \lambda_1}$. Then
\begin{align*}
    x {\leq} \frac{kn}{\lambda_n - \lambda_1} \Longrightarrow f(x,y) = \max \{ f_2, f_4 \},
\end{align*}
which is minimized when $x = kn /(\lambda_n - \lambda_1)$ and $f_2 = f_4$. {Solving $f_2 = f_4$ for $y$, when $x = kn / (\lambda_n - \lambda_1)$, yields $y=0$ and thus $f(\frac{kn}{\lambda_n - \lambda_1},0) = k\vartheta(G)$.} The minimum value of $\lambda_1(X)$, equivalently, the value $\vartheta(K_k \square G)$, thus equals $\min \{ k \vartheta(G), n \}$.

Lastly, by edge-transitivity and vertex-transitivity of $G$, matrices optimal to \ref{eqn:ThetaK_SDP2} for $G$ have a constant row sum. Thus, as \eqref{eqn:ProofOddCycle} shows, \ref{eqn:ThetaK_SDP2} is then equivalent to maximizing the largest (scaled) eigenvalue over feasible matrices. Hence, $\vartheta_k(G) = \min\{ k \vartheta(G), n\}$, as can be shown by derivations similar to those used for \eqref{eqn:ConstantRowSumResult}.
\end{proof}
A graph that is both edge-transitive and vertex-transitive is also known as a symmetric graph. Many Johnson graphs (definition \ref{Def:JohnsonGraphs}) satisfy these properties.

\citeauthor{kuryatnikova2020maximum} \cite{kuryatnikova2020maximum} compute $\vartheta_k(G)$ for several highly symmetric graphs (table 13 in the online supplement to \cite{kuryatnikova2020maximum}). They remark that for those graphs, $\mathrm{\Psi}_{\vartheta_k}(G) = \lceil \vartheta( \overline{G}) \rceil.$ We explain this result for all the graphs present in table 13 except for the graph $H(12,2, \{ i \, | \, 1 \leq i \leq 7 \} )$ (see section \ref{section:HammingGraphs} for the notation). All the other graphs evaluated by \citeauthor{kuryatnikova2020maximum} in table 13 satisfy the assumptions of theorem \ref{thm:CartesianThetaK}, hence, $\vartheta_k(G) = \vartheta(K_k \square G)$ for those graphs. Therefore,
\begin{equation}
    \label{eqn:PsiOperatorEquality}
    \mathrm{\Psi}_{\vartheta_k}(G) = \mathrm{\Psi}_{\vartheta}(G) = \lceil \vartheta( \overline{G}) \rceil.
\end{equation}

Note that the Johnson graph  $J(n,m,  m-1 )$ is regular, vertex-transitive and edge-transitive. Therefore, equation \eqref{eqn:PsiOperatorEquality} holds and $\lceil \vartheta( \overline{ J(n,m, m-1  )} ) \rceil = n-m+1.$

\section{Strongly regular graphs}
\label{section:SRG}
In the previous section we showed that certain classes of graphs allow an analytical computation of $\vartheta_k(G)$. This section expands on the considered classes with strongly regular graphs, see definition \ref{Def:StronglyRegular}.  We also derive analogous  of theorem \ref{thm:CartesianThetaK} for strongly regular graphs and the generalized $\vartheta'$-number, see theorem \ref{thm:CartesianThetaKSRG}.

Let $G$ be a strongly regular graph with parameters $(n, d, \lambda, \mu)$, and $A_G$ its adjacency matrix. Since $G$ is regular with valency $d$, we have that $d$ is an eigenvalue of $A_G$ with eigenvector $\vec{1}$.  The matrix $A_G$ has exactly two
distinct eigenvalues associated with eigenvectors orthogonal to $\vec{1}$. These two eigenvalues are known as restricted eigenvalues and are usually denoted by $r$ and $s$, where $r\geq 0$ and $s\leq -1$.
We consider here connected,  non-complete, strongly regular graphs. For those graphs we have that $s<-1$. Thus, we exclude trivial cases.

Strongly regular graphs attain Lov\'asz bound of theorem \ref{lemma:LovaszThetaUB}, see e.g.,  \citeauthor{haemers1978upper} \cite{haemers1978upper}. In particular, for a strongly regular graph $G$ we have
\[    \vartheta(G) = \frac{n\lambda_n(A_G)}{\lambda_n(A_G) - \lambda_1(A_G)}. \]
In the following theorem we derive an explicit expression for  $\vartheta_k(G)$ for strongly regular graphs.

\begin{theorem} \label{THm:StronglyRegularThetak}
For any strongly regular graph $G$ with parameters $(n,d,\lambda, \mu)$ and restricted eigenvalues $r\geq 0$ and $s< -1$, we  have
\begin{equation*}
    \vartheta_k(G) = \min\{ k \vartheta(G), n \} = \min \Bigg\{ k \frac{n \lambda_n(A_G)}{\lambda_n(A_G) - \lambda_1(A_G)}, n \Bigg \}.
\end{equation*}
\end{theorem}
\begin{proof}
We  prove the result by showing that the lower and upper bound on $\vartheta_k(G)$ coincide. Consider \ref{eqn:ThetaK_SDP2}, and set  $Y = \frac{k}{n} I + xA_{\overline{G}}.$ When $ 0 \preceq Y \preceq I$, $Y$ is feasible for \ref{eqn:ThetaK_SDP2}. These SDP constraints on $Y$ can be rewritten in terms of $x$. As \ref{eqn:ThetaK_SDP2} is a maximization problem we may assume w.l.g.~$x \geq 0$. Thus, for all $i \leq n$,
\begin{equation}
    \label{eqn:SRGproof1}
    \lambda_i(Y) = k/n + x \lambda_i(A_{\overline{G}}).
\end{equation}
Since three eigenvalues of $A_G$ satisfy  $d \geq r > s$, we have
\begin{equation}
    \label{eqn:SRGproof2}
    \sigma(A_{\overline{G}}) = (n-d-1, -1 -s , -1 - r ).
\end{equation}
Substituting \eqref{eqn:SRGproof2} in \eqref{eqn:SRGproof1}  and exploiting the fact that $n-d-1 > -(s+1) > -(1+r)$ we have:
\begin{align*}
    0 \preceq Y \preceq I ~\Leftrightarrow~ 0 \leq \lambda_i(Y) \leq 1 ~\Leftrightarrow~ \begin{cases}
    k/n + x(-1-r) \geq 0 \\
    k/n +x(n-d-1) \leq 1.
    \end{cases}
\end{align*}
The last two inequalities provide upper bounds on $x$, i.e.,
\begin{equation}
    \label{eqn:SRGproof3}
    x \leq \min \Bigg\{ \frac{k}{n(1+r)} , \frac{n-k}{n(n-d-1)} \Bigg\}.
\end{equation}
When $x$ satisfies \eqref{eqn:SRGproof3}, $Y$ is thus feasible for \ref{eqn:ThetaK_SDP2} and $\langle J, Y \rangle$ will provide a lower bound for $\vartheta_k(G)$. In particular, with \eqref{eqn:SRGproof3} at equality,
\begin{equation}
    \label{eqn:SRGproof4}
    \langle J, Y \rangle = k + n(n-d-1)x =\min \Bigg\{ k \Bigg( \frac{r+n-d}{1+r} \Bigg) , n \Bigg\}.
\end{equation}
Equation \eqref{eqn:SRGproof4} implies
\begin{equation}
    \label{eqn:SRGproof5}
    \vartheta_k(G) \geq \min \Bigg\{ k \Bigg( \frac{r+n-d}{1+r} \Bigg) , n \Bigg\}.
\end{equation}
By \eqref{eqn:VarthetaIneq} and proposition \ref{Thm:ChromGeneralLB}, we have $\vartheta_k(G) \leq \min\{  k\vartheta(G), n\}$.
It remains only to show that
\begin{equation}\label{eqn:kthetaRatio}
k \Bigg( \frac{r+n-d}{1+r} \Bigg) =  k\vartheta(G).
\end{equation}
The eigenvalues of $A_G$ can be written in terms of the parameters of $G$, i.e.,
\begin{equation}\label{eqn:SRGproof8}
    rs = \mu - d,  \quad  r+s = \lambda - \mu.
\end{equation}
Furthermore,  the parameters of any strongly regular graph satisfy
\begin{equation}
    \label{eqn:SRGproof9}
    (n-d-1) \mu = d(d - \lambda - 1),
\end{equation}
see e.g., Theorem 9.1.3 in \cite{brouwer2011spectra}. Let us now rewrite the term:
\begin{align}
    \label{eqn:SRGproof10}
    \frac{r+n-d}{1+r}
    = \frac{ns}{s-d}  \frac{(r+n-d)(s-d)}{ns(1+r)}
    =  \frac{ns}{s-d} \frac{ns+nrs+ [d^2-nd -(n-1)sr-d(r+s)]}{ns+nrs},
\end{align}
and evaluate the expression between the square brackets by using \eqref{eqn:SRGproof8} and \eqref{eqn:SRGproof9}, i.e.,
\begin{align*}
    d^2-nd -(n-1)rs-d(r+s) &= d\lambda + d + (n-d-1) \mu -nd -(n-1)(\mu - d) -d(\lambda - \mu)=0.
\end{align*}
Thus \eqref{eqn:SRGproof10} equals $ns/(s-d)$ and $ns/(s-d)=\vartheta(G)$, which proves the theorem.
\end{proof}
Recall that in section \ref{Section:AnalysisOfCart}  we consider symmetric graphs (graphs that are both edge-transitive and vertex-transitive). Although many graphs belong to both symmetric and strongly regular
 classes,  note that neither one is a subset of the other.
The graph $C_6$ is an example of a graph that is symmetric, but not strongly regular. The strongly regular Chang graphs (\citeauthor{chang1959uniqueness} \cite{chang1959uniqueness}) provide an example of a strongly regular graph which is not symmetric.

In section \ref{Section:AnalysisOfCart} we have proved that   $\vartheta(K_k \square G) =  \min \{ k \vartheta(G) , n \}=\vartheta_k(G)$ holds for symmetric graphs, see theorem \ref{thm:CartesianThetaK}.  We show below that a similar relation holds also for strongly regular graphs. In fact we prove a result for the generalized $\vartheta'$-number, denoted by $\vartheta'_k(G)$, that is the optimal value of the SDP relaxation  \ref{eqn:ThetaK_SDP2} strengthened  by adding  non-negativity constraints on the matrix variable. The generalized $\vartheta'$-number for $k=1$  is also known  as the Schrijver's number.

To prove our result, we first present an SDP relaxation that relates $\vartheta'(K_k \square G)$ and $\vartheta_k(G)$. \citeauthor{kuryatnikova2020maximum}  \cite{kuryatnikova2020maximum} introduce the following SDP  relaxation
\begin{eqnarray}
\theta_k^3(G) = &  {\underset{Z\in \mathbb{S}^{n}}{ \text{ Maximize}}}  &   \langle I, Y \rangle  \nonumber   \\
   & \text{subject to} &  Y_{ij} = 0   \quad  \forall (i,j) \in E(G)  \nonumber  \\
   &&  Y_{ii} \le 1  \quad \forall i\in [n]  \nonumber \\
   &&   \begin{bmatrix} k & {\rm diag}(Y)^\top \\{\rm diag}(Y) & Y \end{bmatrix} \succeq 0, \ Y\ge 0,   \nonumber
\end{eqnarray}
that provides an upper bound for $\alpha_k(G)$, the optimal value for the M$k$CS problem.
The above relaxation can be simplified when $G$ is a highly symmetric graph. In particular, if $G$ is a  strongly regular graph one can restrict optimization of the above SDP relaxation to feasible points in the coherent algebra spanned by $\{ I, A, J-I-A \}$. By applying symmetry reduction, the above SDP relaxation reduces to the following  convex optimization problem:
\begin{subequations}
\begin{eqnarray}
\theta_k^3(G)  :=& \text{ Maximize}  & n y_1 \label{obj2}  \\
 & \text{subject to} &  y_1 +(n-d-1)y_2 - \frac{n}{k}y_1^2 \geq 0 \label{M2c1}  \\
& & y_1-(r+1)y_2 \geq 0 \label{M2c2}\\
& & y_1-(s+1)y_2 \geq 0 \label{M2c3}\\
& & y_1 \leq 1 \label{M2c4} \\
& & y_1,y_2 \geq 0. \label{M2c5}
\end{eqnarray} \label{matrSDP}
\end{subequations}
For details on symmetry reduction see e.g., \cite{Parillo,kuryatnikova2020maximum} and references therein.

In \cite{kuryatnikova2020maximum} the authors conjecture that $\theta_k^3(G) \leq \vartheta_k'(G)$ for any graph $G$. Here we show that  $\theta_k^3(G) = \vartheta_k'(G)$ for any (non-trivial) strongly regular graph $G$.
\begin{lemma} \label{lem:Matrixliftingsrg}
Let $G$ be a strongly regular graph with parameters $(n,d,\lambda, \mu)$ and restricted eigenvalues $r\geq 0$ and $s< -1$.
Then
\[
\theta_k^3(G)  = \min \left \{ k \left ( \frac{r+n-d}{r+1} \right ), n  \right \} = \vartheta_k(G) = \vartheta_k'(G).
\]
\end{lemma}
\begin{proof}
Note that for $s< -1$  constraint \eqref{M2c3} is trivially satisfied.
Points in which constraints \eqref{M2c1} and \eqref{M2c2} intersect are $(0,0)$ and $\left (\frac{k(r+n-d)}{n(r+1)},\frac{k(n+r-d)}{n(r+1)^2}\right )$. The  first equality follows by combining the latter point  and  constraint \eqref{M2c4}.
 The second equality follows from  $\vartheta(G)=(r+n-d)/(1+r)$, see \eqref{eqn:kthetaRatio}, and theorem \ref{THm:StronglyRegularThetak}.  The third equality follows from \eqref{eqn:SRGproof3} and the fact  that  $\frac{k}{n(1+r)}\geq 0$ and $\frac{n-k}{n(n-d-1)} \geq 0$.
\end{proof}
It is known that $\vartheta'(K_k \Box G)  \leq  \theta_k^3(G),$ see section 5.1 in \cite{kuryatnikova2020maximum}.
We show below that  equality holds when $k<n(r+1)/(r+n-d)$ and $G$ is a (non-trivial) strongly regular graph by proving an equivalence between the SDP relaxations that give  $\vartheta'(K_k \Box G)$ and  $\theta_k^3(G)$.
The SDP relaxation  for $\vartheta'(K_k \Box G)$, see also \ref{eqn:ThetaK_SDP2}, is invariant under permutations of $k$ colors when the graph under the consideration is $K_k \Box G$. This was exploited in \cite{kuryatnikova2020maximum} to derive the following symmetry reduced relaxation:
\begin{eqnarray*}
\vartheta'(K_k \Box G)  =&   {\underset{X, Z\in \mathbb{S}^{n}}{ \text{ Maximize}}} & \ \langle I, X \rangle \\
   & \text{subject to} \hspace{0.3cm}   &  X_{ij} = 0 \quad   \forall (i,j) \in E(G)  \\
  & &  Z_{ii} = 0 \quad    \forall i \in [n]  \\
  & &  X\ge 0, \ Z\ge 0, \ X-Z\succeq 0   \\
 &&  \begin{bmatrix} 1 & {\rm diag}(X)^\top  \\[1ex]  {\rm diag}(X) & X+(k-1)Z \end{bmatrix}   \succeq 0. \\
\end{eqnarray*}
The above relaxation can be further simplified when $G$ is a strongly regular graph.
One can restrict optimization to the corresponding  coherent algebra.
By applying symmetry reduction, the above SDP relaxation reduces to the following  optimization problem:
\begin{subequations}
\begin{align}
\vartheta'(K_k \Box G) :=&  \text{ Maximize} && nx_1 \label{obj3} \\
& \text{subject to} & & x_1 + (n- d-1)x_2 - (d z_1+(n-d-1)z_2) \geq 0 \label{M3c1}  \\
&& & x_1 - (r+1)x_2-(rz_1-(r+1)z_2) \geq 0 \label{M3c2}\\
&& & x_1 - (s+1)x_2 - (sz_1 -(s+1)z_2) \geq 0 \label{M3c3}\\
&& & x_1 + (n-d-1)x_2+(k-1)(d z_1 +(n-d-1)z_2) -nx_1^2\geq 0 \label{M3c4}\\
&& & x_1 -(r+1)x_2 +(k-1)(rz_1-(r+1)z_2) \geq 0 \label{M3c5}\\
&& & x_1 - (s+1)x_2 +(k-1)(sz_1-(s+1)z_2) \geq 0 \label{M3c6}\\
&& & x_1 \leq 1 \label{M3c7} \\
&& &  x_1,x_2,z_1,z_2 \geq 0 \label{M3c8}.
\end{align} \label{vectorThetaP}
\end{subequations}
Our next result relates optimization problems  \eqref{matrSDP}  and \eqref{vectorThetaP}.
\begin{proposition} \label{prop:EquivalentSrg}
Let $G$ be a strongly regular graph with parameters $(n,d,\lambda, \mu)$ and restricted eigenvalues $r\geq 0$ and $s< -1$, and  $k<\frac{n(r+1)}{r+n-d}$.
Then the optimization problems \eqref{matrSDP}  and \eqref{vectorThetaP}  are equivalent.
\end{proposition}
\begin{proof}
Let  $(x_1,x_2,z_1,z_2)$ be feasible for  \eqref{vectorThetaP}.
 We show that $(y_1,y_2)$ where  $y_1:=x_1$ and $y_2:=x_2$ is feasible for   \eqref{matrSDP}.

From \eqref{M3c1} and \eqref{M3c4} we have
 \begin{equation*}
  \begin{cases}
    x_1 + (n-d-1)x_2 \geq (d z_1+(n-d-1)z_2)\\
    x_1 + (n-d-1)x_2 \geq nx_1^2 -(k-1)(d z_1 +(n-d-1)z_2),
  \end{cases}
\end{equation*}
from where it follows
$$ x_1 + (n-d-1)x_2 -\frac{n}{k}x_1^2  \geq \max \left \{(d z_1+(n-d-1)z_2) - \frac{n}{k}x_1^2, (k-1)(\frac{n}{k}x_1^2-(d z_1 +(n-d-1)z_2) ) \right \}. $$
To verify that the right hand side above is non-negative, note that either $d z_1 +(n-d-1)z_2 \geq \frac{n}{k}x_1^2$  or $d z_1 +(n-d-1)z_2 < \frac{n}{k}x_1^2$. Therefore $ x_1 + (n-d-1)x_2 -\frac{n}{k}x_1^2  \geq 0$ and constraint \eqref{M2c1} is satisfied.

Similarly, from \eqref{M3c2} and \eqref{M3c5} it follows that constraint \eqref{M2c2} is satisfied. Constraint \eqref{M2c3} is trivially satisfied by \eqref{M3c3} and \eqref{M3c6}.

Conversely, let $(y_1,y_2)$ be feasible for \eqref{matrSDP}. Define  $x_1:=y_1$ and $x_2:=y_2$. Let $z_1$ and $z_2$ be the solutions of the following system of equations:
\begin{align*}
rz_1 = (r+1)z_2, \quad dz_1 + (n-d-1)z_2 &= \frac{n}{k}x_1^2.
\end{align*}
Thus, $z_1=\frac{n(r+1)}{k(d+r(n-1))} x_1^2$, $z_2=z_1 \frac{r}{r+1}$. Therefore, constraint  \eqref{M3c1}  follows from \eqref{M2c1}  and the construction of $z_1$ and $z_2$. Similar arguments applied to \eqref{M2c2} can be used to verify that \eqref{M3c2} and \eqref{M3c5} are satisfied. To verify \eqref{M3c4} we rewrite the constraint as follows
\begin{align*}
x_1 + (n-d-1)x_2 + & (k-1)(d z_1 +(n-d-1)z_2)-nx_1^2 = \\
  & x_1 + (n-d-1)x_2 - \frac{n}{k}x_1^2 + (k-1)\left (d z_1 +(n-d-1)z_2 -  \frac{n}{k}x_1^2 \right ) \geq 0.
\end{align*}
To verify constraint \eqref{M3c3} we exploit the construction of $z_1$ and $z_2$ as well as $r\geq 0$ and $s<-1$ to obtain:
$-(sz_1 -(s+1)z_2) =  \frac{r-s}{r+1} z_1 \geq 0.$ It remains to show that constraint \eqref{M3c6} is redundant for $k<\frac{n(r+1)}{r+n-d}$.
Let us rewrite the constraint as follows
\begin{align} \label{con:redundantParabola}
 x_1 - (s+1)x_2 +(k-1)(sz_1-(s+1)z_2) = x_1  - (s+1)x_2 - \frac{n(k-1)(r-s)}{k(d+r(n-1)) } x_1^2 \geq 0.
\end{align}
A point of intersection of $x_1  - (s+1)x_2 -  \frac{n(k-1)(r-s)}{k(d+r(n-1)) } x_1^2 = 0$ and
$x_1=(r+1)x_2$ is $\left (\frac{k(d+r(n-1))}{n(k-1)(r+1)},\frac{k(d+r(n-1))}{n(k-1)(r+1)^2} \right )$, and a point of intersection of $x_1 +(n-d-1)x_2 - \frac{n}{k}x_1^2 =0 $ and $x_1=(r+1)x_2$ is $\left (\frac{k(r+n-d)}{n(r+1)},\frac{k(r+n-d)}{n(r+1)^2} \right)$.
Furthermore, an intersection point of
$x_1 +(n-d-1)x_2 - \frac{n}{k}x_1^2 =0 $ and  the $x_1$-axis is $(\frac{k}{n},0)$, and a point of intersection of $x_1  - (s+1)x_2 -  \frac{n(k-1)(r-s)}{k(d+r(n-1)) } x_1^2 = 0$ and the $x_1$-axis is $\left (\frac{k(d+r(n-1))}{n(k-1)(r-s)},0 \right)$.
Note that the common intersection point of both parabolas,  $x_1=(r+1)x_2$ and the $x_1$-axis is $(0,0)$.
Let us find $k$ for which
$$\frac{k(r+n-d)}{n(r+1)}< \frac{k(d+r(n-1))}{n(k-1)(r+1)}
~~\Leftrightarrow ~
k<\frac{n(r+1)}{r+n-d}
$$
and
$$\frac{k}{n}< \frac{k(d+r(n-1))}{n(k-1)(r-s)}
~~\Leftrightarrow ~
k < \frac{d+rn-s}{r-s}.
$$
By using \eqref{eqn:SRGproof10}, one can verify that  $\frac{n(r+1)}{r+n-d}< \frac{d+rn-s}{r-s}$, from where it follows that the constraint \eqref{con:redundantParabola} is redundant when $\frac{k(r+n-d)}{n(r+1)}<1$.

It follows trivially that the objective values coincide for feasible solutions of two models that are related as described.
\end{proof}
Now, from the previous discussion it follows the next result.
\begin{theorem}
\label{thm:CartesianThetaKSRG}
Let $G$ be a strongly regular graph with parameters $(n,d,\lambda, \mu)$ and restricted eigenvalues $r\geq 0$ and $s< -1$,  and  $k<\frac{n(r+1)}{r+n-d}$.
Then $\vartheta'(K_k \square G) =   \vartheta'_k(G).$
\end{theorem}
\begin{proof}
The proof follows from lemma \ref{lem:Matrixliftingsrg} and proposition \ref{prop:EquivalentSrg}.
\end{proof}

\section{Orthogonality graphs}
\label{section:OrthogonalityGraph}
In this section we compute the generalized $\vartheta$-number for  the orthogonality graphs.

We motivate the study of orthogonality graphs by a scenario, taken from \citeauthor{galliard2003impossibility} \cite{galliard2003impossibility}. Let $n=2^r$ for some $r \geq 1$. Consider a game where two players, Alice and Bob, each receive an $n$-dimensional binary vector as input. These vectors are either equal or their Hamming distance (see definition \ref{Def:HammingDistance}) equals $n/2$, that is, they differ in exactly $2^{r-1}$ positions. Given these inputs, Alice and Bob must each return a $r$-dimensional binary vector as output. To win the game, Alice and Bob must return equal outputs if and only if their inputs were equal. Alice and Bob are not permitted to communicate once they receive their inputs. The players are, however, allowed to coordinate a strategy beforehand.
One such strategy results in the  definition of an orthogonality graph.

Vertices of the orthogonality graph $\Omega_n$ are represented by all the unique $n$-dimensional binary vectors. Vertices (equivalently vectors) are adjacent if their Hamming distance equals $n/2$ and thus, $\Omega_n = H(n, 2, \{ n / 2 \} )$. Here  $H(n, 2, \{ n / 2 \} )$ denotes the Hamming graph, see definition \ref{Def:HammingGraph}.

The strategy of Alice and Bob then comprises graph coloring for $\Omega_n$ before the game starts. After being given their input vector, Alice and Bob should respond the color of their vector, encoded as $r$-dimensional binary vector. With this $r$-dimensional vector, Alice and Bob can indicate $2^r = n$ distinct colors. Disregarding any luck in guessing, the game can always be won if and only if $\chi(\Omega_n) \leq n$.

The orthogonality graph gets its name from another description of the graph, that is, when the vectors have $\{\pm 1\}$ entries. The Hamming distance between two binary vectors of $n/2$ then corresponds to those $\{
\pm 1\}$ vectors being orthogonal to each other.

\citeauthor{godsil2008coloring} \cite{godsil2008coloring} prove that $\chi(\Omega_{2^r}) = 2^r$ for $r \in \{1, 2, 3\}$ and $\chi(\Omega_{2^r}) > 2^r$ otherwise. This means that the game can only be won for $r \leq 3$.

Clearly, for odd $n$, $\Omega_n$ is edgeless. We therefore restrict the analysis of $\Omega_n$ to the case when $n$ is a multiple of 4. When that is the case, $\Omega_n$ consists of two isomorphic components, for vectors of even and odd Hamming weights respectively.

Next to $\chi(\Omega_n)$, the independence number $\alpha(\Omega_n)$ has been studied in multiple papers. The two graph parameters are related by $|V| \leq \chi(G)\alpha(G)$, for any graph $G = (V, E)$, see \eqref{eqn:ChromSimpleLBvertices}.
\citeauthor{frankl1986} \cite{frankl1986} and
\citeauthor{galliard2001classical} \cite{galliard2001classical}
constructed  a stable set of
$\Omega_n$ of size $\underline{\alpha}(n)$ for $n \equiv 0\mod 4$. In particular,
\begin{equation}
\label{eqn:OrthStableSetConstr1}
\alpha(\Omega_n) \geq \underline{\alpha}(n) = 4 \sum_{i=0}^{n/4-1} \binom{n-1}{i}.
\end{equation}
On the other hand \citeauthor{de2007note} \cite{de2007note} used an SDP relaxation to find $\alpha(\Omega_{16}) = \underline{\alpha}(16) = 2306$.
It is also known that  $\alpha(\Omega_{24}) = \underline{\alpha}(24) = 178208$, see \cite{ihringer2019independence}. 
More recently, \citeauthor{ihringer2019independence} \cite{ihringer2019independence} have proven $\alpha(\Omega_{2^r}) = \underline{\alpha}(2^r)$ for $r \geq 2$. The {\it conjecture} by \citeauthor{godsil2008coloring} \cite{godsil2008coloring} (see also \citeauthor{franklRodl1987} \cite{franklRodl1987}) whether
\begin{align}\label{conjecture}
\alpha(\Omega_{4m}) = \underline{\alpha}(4m) \mbox{  for $m \geq 1$}
\end{align}
remains an open problem.

We proceed by computing $\vartheta_k (\Omega_n)$ when $n$ is a multiple of four. From \citeauthor{newman2004independent} \cite{newman2004independent}, the (unordered) eigenvalues of $\Omega_n$ are then given by
\begin{equation*}
    \lambda_r = \frac{2^{n/2}}{(n/2)!} \prod_{i=1}^{n/2} (2i-1-r), \quad 1 \leq r \leq n,
\end{equation*}
and $\lambda_0 = \binom{n}{n/2}$ (since $\Omega_n$ is regular with degree $\lambda_0$).
The smallest eigenvalue is obtained for $r=2$ and thus
\begin{equation*}
    \lambda_2 = \frac{1}{1-n} \binom{n}{n/2}.
\end{equation*}
Since $\Omega_n$ is isomorphic to a binary Hamming graph, $\Omega_n$ is a symmetric graph. The bound of theorem~\ref{Thm:EdgeTransitiveThetaK} thus holds with equality. \citeauthor{newman2004independent} \cite{newman2004independent} also shows that the multiplicity of $\lambda_2$ equals $n^2 - n$. This multiplicity exceeds $n$ since $n$ is a multiple of $4$. Then it is not hard to show (by a method comparable to the one used in the proof of theorem \ref{Thm:KneserGraphThetak}) that
\begin{equation*}
\vartheta_k(\Omega_n) = k \frac{2^n}{n}, \quad k \leq n.
\end{equation*}
When $k=1$, $\vartheta_k(\Omega_n)$ coincides here with the so called  ratio bound. This bound refers to \eqref{eqn:GodsilNewmanIndependentSetBoun} for regular graphs and was also computed for $\Omega_n$ in \cite{newman2004independent}.

Let $S$ be a stable set of size $\underline{\alpha}(n)$ that contains no vectors that have their Hamming weight contained in $W = \{ n/4+1$, $n/4+3$, $\ldots$, $3n/4-1 \}$. Furthermore, note that the Johnson graphs (see definition \ref{Def:JohnsonGraphs}) appear as induced subgraphs of $\Omega_n$. Let $w \in W$ and consider the subgraph of $\Omega_n$ induced by $J(n,w, w - n/4 )$. This subgraph contains no vertices in $S$ and thus
\begin{equation}
    \alpha_2(\Omega_n) \geq \underline{\alpha}(n) + 4 \underset{w \in W}{\max} \{ \alpha( J(n,w, w - n/4)) \}.
\end{equation}
We may multiply the independence number of $J(n,w,w - n/4)$ by 4 since we can take bitwise complements and find an isomorphic stable set in the isomorphic second component of $\Omega_n$.

In  section \ref{section:HammingGraphs} we prove that $\chi_k(\Omega_{4n+2}) = 2k$.

\section{New bounds on $\chi_k(G)$}
\label{sect:lower_bound_for_multicoloring}
In this section we first we derive bounds on the product and sum of $\chi_k({G})$ and $\chi_k(\overline{G})$. Then, we provide graphs for which the bounds are sharp. Lastly, we
derive  spectral lower bounds on the multichromatic number of a graph.

A famous result by \citeauthor{nordhaus1956complementary} \cite{nordhaus1956complementary} states that
\begin{align}
    &n \leq \chi(G) \chi(\overline{G}) \leq \Big( \frac{n+1}{2} \Big)^2, \label{Nordhaus1} \\
    &2\sqrt{n} \leq \chi(G) + \chi( \overline{G} ) \leq n+1. \label{Nordhaus2}
\end{align}
Various papers have been published on determining Nordhaus–Gaddum inequalities  for other graph parameters, such as the independence
and edge-independence number (see \cite{aouchiche2013survey} for a survey). We provide Nordhaus–Gaddum inequalities for $k$-multicoloring.
\begin{theorem}
\label{Thm:NordhausMulticoloring}
For any graph $G = (V,E)$, $|V| = n$, we have
\begin{align*}
    &k^2n \leq \chi_k(G) \chi_k(\overline{G}) \leq k^2 \Big( \frac{n+1}{2} \Big)^2, \\
    &2k \sqrt{n} \leq \chi_k(G)  + \chi_k(\overline{G}) \leq k(n+1).
\end{align*}
\end{theorem}
\begin{proof}
We follow the original proof as given by \citeauthor{nordhaus1956complementary} \cite{nordhaus1956complementary}, extended to the $k$-multicoloring case. Consider an optimal $k$-multicoloring of $G$, using $\chi_k(G)$ colors. Then for $i = 1, 2, \ldots, \chi_k(G)$, define $n_i$ as the set of vertices that are colored with color $i$. We have $\sum_{i=1}^{\chi_k(G)} |n_i| = nk.$ Furthermore,
\begin{equation}
    \label{eqn:Nordhaus1}
    \max |n_i| \geq \frac{nk}{\chi_k(G)}.
\end{equation}
Consider the largest set $n_i$. Since the vertices in this set share a color, they form a stable set in $G$. Thus they form a clique in $\overline{G}$. Accordingly,
\begin{equation}
    \label{eqn:Nordhaus2}
    \chi_k(\overline{G}) \geq k \omega(\overline{G}) \geq k \max |n_i|.
\end{equation}
Combining \eqref{eqn:Nordhaus1} and \eqref{eqn:Nordhaus2} proves the lower bound on the product of  $\chi_k({G}) $ and $\chi_k(\overline{G})$.

The  lower bound on the sum of  $\chi_k({G}) $ and $\chi_k(\overline{G})$ can be proven by algebraic manipulation:
\begin{align*}
    &(\chi_k(G) - \chi_k(\overline{G}))^2 \geq 0 ~\Longrightarrow~ \chi_k(G)^2 + \chi_k(\overline{G})^2 + 2 \chi_k(G) \chi_k(\overline{G}) \geq 4 \chi_k(G) \chi_k(\overline{G}) \nonumber \\
    & \Longrightarrow \chi_k(G) + \chi_k(\overline{G}) \geq 2 \sqrt{\chi_k(G) \chi_k(\overline{G})} \geq 2k \sqrt{n} \nonumber.
\end{align*}
The two upper bounds can be proven by combining \eqref{eqn:MultichromBounded}, \eqref{Nordhaus1}  and  \eqref{Nordhaus2} as follows:
\begin{align*}
    &\chi_k(G) \chi_k(\overline{G}) \leq k^2 \chi(G) \chi(\overline{G}) \leq k^2 \Big( \frac{n+1}{2} \Big) ^2, \\
    & \chi_k(G)  + \chi_k(\overline{G}) \leq k (\chi(G) + \chi(\overline{G})) \leq k(n+1).
\end{align*}
\end{proof}
The second upper bound in theorem \ref{Thm:NordhausMulticoloring} can also be found in \cite{brigham1982generalized} (in a slightly generalized form).
We present below graphs for which  the bounds in theorem \ref{Thm:NordhausMulticoloring} are attained. For that purpose we define the graph sum of two graphs. The graph sum of graphs $G_1$ and $G_2$ is the graph, denoted by $G_1 + G_2$, whose vertices  and edges are defined as follows:
\begin{align*}
     V(G_1 + G_2) := V(G_1)  \cup V(G_2), \ E(G_1 + G_2) := E(G_1) \cup E(G_2).
\end{align*}
\citeauthor{nordhaus1956complementary} \cite{nordhaus1956complementary} show that the upper bounds in their theorem are attained by graph $G = K_p + \overline{K}_{p-1}$. Graph $G$ has $n = 2p-1$ vertices. It is clear that $\chi(G) = \chi(\overline{G}) = p = \frac{n+1}{2}$. Thus $G$ attains both upper bounds simultaneously. As both $G$ and $\overline{G}$ are weakly perfect graphs, we can apply \eqref{eqn:WeaklyPerfectMulticoloring} to find $\chi_k(G) = \chi_k(\overline{G}) = kp = k \frac{n+1}{2}$. This implies that graph $G$ also attains the upper bounds in theorem \ref{Thm:NordhausMulticoloring}.
\citeauthor{nordhaus1956complementary} \cite{nordhaus1956complementary} also provide an example of a graph which attains the lower bounds in their theorem. This example extends to the multichromatic variant as well. Let $m_1 = m_2 = \ldots = m_p = p$ and consider the complete multipartite graph $G = K_{m_1, \ldots, m_p}$. Then $\chi_k(G) = \chi_k(\overline{G}) = kp = k \sqrt{n}$. Thus, this graph $G$ attains the lower bounds in theorem \ref{Thm:NordhausMulticoloring}.
In fact, for any graph $G$ such that $\chi(G)\chi(\overline{G}) = |V(G)|$, we  have $ k^2 |V(G)| \leq \chi_k(G)\chi_k(\overline{G})$ by theorem \ref{Thm:NordhausMulticoloring}, and $\chi_k(G)\chi_k(\overline{G}) \leq k^2 \chi(G)\chi(\overline{G}) = k^2|V(G)|$ by \eqref{eqn:MultichromBounded}. Since the upper and lower bound coincide, we have $\chi_k(G)\chi_k(\overline{G}) = k^2 |V(G)|$. The set of vertex-transitive graphs provides a number of examples for which this bound is attained, such as the Johnson graph $J(n,2,1)$ when $n$ is even.

The chromatic number of a graph is bounded by the spectrum of matrices related to its adjacency matrix.
This well known result is given below.
\begin{theorem}[\citeauthor{Hoffman70} \cite{Hoffman70}] \label{lemma:HoffmanChromLB}
If $G$ has at least one edge, then
$\chi(G) \geq 1 - \frac{\lambda_1(A_G)}{ \lambda_n(A_G)}.$
\end{theorem}
Since each color class has size at most $\alpha(G)$, we have that
\begin{equation}
    \label{eqn:ChromSimpleLBvertices}
    \chi(G) \geq \frac{n}{\alpha(G)},
\end{equation}
where $n$ is the number of vertices in $G$.
Therefore one can use upper bounds for  $\alpha(G)$ to derive lower bounds for $ \chi(G)$.
From \eqref{eqn:IndepNumberLexiProd} it follows that $\alpha(G \circ K_k) = \alpha(G)$. Thus we can establish the multicoloring variant of \eqref{eqn:ChromSimpleLBvertices}:
\begin{equation}
    \label{eqn:SimpleMulticoloringVerticesLB}
    \chi_k(G) = \chi(G \circ K_k) \geq \frac{| V(G \circ K_k) | }{\alpha(G \circ K_k)} = k \frac{n}{\alpha(G)}.
\end{equation}
Note that the above result also follows from \eqref{eqn:Nordhaus1}.
The bound \eqref{eqn:SimpleMulticoloringVerticesLB} is also given in \cite{campelo2013optimal}, where the authors show that the lower bound is tight for webs and antiwebs.
Note that for a graph $G$ such that $\alpha_k(G) = k\vartheta(G)$ we have that
$\alpha_k(G) = k\alpha(G)$, see lemma 5 in \cite{kuryatnikova2020maximum}, and thus $\chi_k(G)  \geq  \frac{k^2 n}{\alpha_k(G)}.$
The above inequality is satisfied for example for  the Johnson graph $J(n,2,1)$ when $n$ is even, and for  $J(n,3,2)$ when $v \equiv$ 1 or 3 mod 6.

Let us now present  known upper bounds for the independence number of a graph.
\begin{theorem}[\citeauthor{Hoffman70} \cite{Hoffman70}]
\label{lemma:HoffmanStableSetUB}
For any $d$-regular graph $G$ of order $n$, we have $\alpha(G) \leq n \frac{\lambda_n(A_G)}{\lambda_n(A_G) -d}.$
\end{theorem}
The result of theorem \ref{lemma:HoffmanStableSetUB} applies only to regular graphs with no loops.
\citeauthor{Haemers80}  generalizes the Hoffman bound as follows.
\begin{theorem}[\citeauthor{Haemers80} \cite{Haemers80}] \label{thm:Haemenrs80}
Let $G$ have minimum vertex degree $\delta$. Then
\begin{equation}  \label{eqn:HaemersIndependentSetBoun}
\alpha(G) \leq n \frac{ \lambda_1(A_G) \lambda_n(A_G) }{ \lambda_1(A_G) \lambda_n(A_G)-\delta^2}.
\end{equation}
\end{theorem}
If $G$ is regular, then the result of theorem \ref{thm:Haemenrs80} reduces to Hoffman's bound.
Another extension of the bound of Hoffman  is given by \citeauthor{godsil2008eigenvalue}.
\begin{theorem}[\citeauthor{godsil2008eigenvalue} \cite{godsil2008eigenvalue}]
Let $G$ be a loopless graph and  $L_G$ its Laplacian matrix.  Then
\begin{equation}
    \label{eqn:GodsilNewmanIndependentSetBoun}
    \alpha(G) \leq n \frac{ \lambda_1(L_G) - \overline{d}_G}{ \lambda_1(L_G) },
\end{equation}
where $\overline{d}_G$ denotes the average degree of the  vertices of $G$.
\end{theorem}

Now we are ready to present our results.
\begin{lemma}\label{thm:MulticoloringSpectrumLB2}
Let $G$ have minimum vertex degree $\delta$. Then
\begin{equation*}
    \chi_k(G) \geq k \frac{ \lambda_1(A_G) \lambda_n(A_G)-\delta^2 }{\lambda_1(A_G) \lambda_n(A_G)}.
\end{equation*}
\end{lemma}
\begin{proof}
The result follows by combining \eqref{eqn:SimpleMulticoloringVerticesLB} and \eqref{eqn:HaemersIndependentSetBoun}.
\end{proof}
\begin{lemma}
\label{thm:MulticoloringSpectrumLB}
For any loopless  graph $G$, we have
\begin{equation*}
    \chi_k(G) \geq k \frac{\lambda_1(L_G)}{\lambda_1(L_G) - \overline{d}_G},
\end{equation*}
where $\overline{d}_G$ denotes the average degree of its vertices, and $L_G$ the Laplacian matrix of $G$.
\end{lemma}
\begin{proof}
The result follows by combining   \eqref{eqn:SimpleMulticoloringVerticesLB} and \eqref{eqn:GodsilNewmanIndependentSetBoun}.
\end{proof}

When $G$ is a regular graph, \eqref{eqn:GodsilNewmanIndependentSetBoun} is equivalent to the result of theorem \ref{lemma:HoffmanStableSetUB}, and therefore the result of lemma \ref{thm:MulticoloringSpectrumLB} is equivalent to:
\begin{equation*}
    \chi_k(G) \geq k \Bigg( 1 - \frac{\lambda_1(A_G)}{\lambda_n(A_G)} \Bigg).
\end{equation*}
It is not difficult to verify that complete graphs  attain the bound of lemma \ref{thm:MulticoloringSpectrumLB2} and lemma \ref{thm:MulticoloringSpectrumLB}.\\

We end this section by presenting bounds on the multichromatic number of Johnson graphs, see definition \ref{Def:JohnsonGraphs}.
We study the simple case $J(n,2, 1)$, $n \geq 4$. Graph $J(n,2,1)$ is sometimes referred to as the triangular graph. The graph $J(n,2, 1)$ is the complement graph of the Kneser graph $K(n,2)$, and both are known to be strongly regular. Every vertex of $J(n,2,1)$ corresponds to a set of two elements. These two elements can be thought of as two vertices of the complete graph $K_n$, with the vertex in $J(n,2,1)$ representing the edge between these two vertices of $K_n$. Graph $J(n, 2, 1)$ is thus the line graph of the complete graph $K_n$. For any graph $G$, its line graph is denoted $L(G)$.

\begin{proposition} \label{prop:chiTriangular}
For the triangular graph  $J(n,2, 1)$ and $n \geq 4$, we have
\begin{equation*}
k(n-1) \leq \chi_k(J(n,2,  1)) \leq k ( 2 \floor{\frac{n-1}{2}}+1).
\end{equation*}
\end{proposition}

\begin{proof}
As $J(n,2,1)$ is isomorphic to $L(K_n)$, a coloring of $J(n,2, 1)$  is equivalent to an edge coloring of $K_n$. It is not hard to see that $\omega(L(G)) $ equals the maximum degree of a vertex of $G$. Thus $\omega(L(K_n)) = n-1$.
For even $n$, $\chi(L(K_n)) = n-1$, see \citeauthor{baranyai1974factorization} \cite{baranyai1974factorization}. Therefore, for even $n$ we have
\begin{equation*}
    \chi(L(K_n)) = \omega(L(K_n)) \Longrightarrow \chi_k(L(K_n)) = k \chi(L(K_n)) = k(n-1),
\end{equation*}
where the implication follows from \eqref{eqn:MultichromBounded}. For odd $n$, $\chi(L(K_n)) = n$, see \citeauthor{vizing1964estimate} \cite{vizing1964estimate}. By \eqref{eqn:MultichromBounded}, the proposition follows.
\end{proof}
Note that in the proof of the previous proposition we could also exploit the following well known result: $\alpha(K(n,2))=n-1$.

\subsection{Hamming graphs}
\label{section:HammingGraphs}
In this section we present results for the (multi)chromatic number of Hamming graphs (definition \ref{Def:HammingGraph}). We also provide sufficient and necessary conditions for the Hamming graph to be perfect.

In the Hamming graph $H(n, q, F)$, the vertex set is the set of $n$-tuples of letters from an alphabet of size $q$, and vertices $u$ and $v$ are adjacent if  their Hamming distance satisfies $d(u,v) \in F$.
Note that $|V(H(n, q, F))|=q^n$.  By slight abuse of notation, we will use the terms vectors and vertices interchangeably, as they permit a one-to-one correspondence in Hamming graphs. Many authors refer to $H(n,q,\{1\})$ as the Hamming graph. The graph $Q_n:=H(n,2,\{ 1\})$ is also known as the binary Hamming graph or hypercube graph.

We first list several known results for $H(n,q, \{1\})$.  Graph $H(n,q, \{1\})$ equals the Cartesian product of $n$ copies of $K_q$. Thus $H(n,q, \{1\}) = \square^n K_q$, see definition \ref{Def:GraphProducts}. Furthermore, it holds   $\chi(G_1 \square G_2) = \text{max} \{ \chi(G_1), \chi(G_2) \}$, see  \citeauthor{sabidussi1957graphs} \cite{sabidussi1957graphs}. Therefore,
$\chi(H(n,q, \{1\})) = q.$
To derive the independence number of $H(n,q, \{1\})$, we proceed as follows. Let $S \subset V$ be a stable set of $H(n,q, \{1\})$. Then
$\underset{u,v \in S, \, u \neq v}{\text{Min}} d(u,v) \geq 2.$
From coding theory, the Singleton bound (\citeauthor{singleton1964maximum} \cite{singleton1964maximum}) is an upper bound on the maximum number of codes of length $n$, using an alphabet of size $q$,
such that a Hamming distance between any two codes is at least  two. In particular, from the Singleton bound we have   $\alpha(H(n,q, \{1\})) \leq q^{n-1}$.

To show that $\alpha(H(n,q, \{1\})) \geq q^{n-1}$, we construct an independent set in the Hamming graph of size $q^{n-1}$, by a construction employed in \cite{singleton1964maximum}. Consider all the vectors in $(\mathbb{Z}/q\mathbb{Z})^n$ for which
the coordinates sum to some $x \in \mathbb{Z}/q\mathbb{Z}$. By symmetry, there exist $q^{n-1}$ vectors satisfying this condition. Note that any two different vectors satisfying this condition must differ in at least two positions, which implies they are not adjacent. Thus, $\alpha(H(n,q, \{1\})) \geq q^{n-1}$ and combined with the Singleton bound, this gives $\alpha(H(n,q, \{1\})) = q^{n-1}$.

To the best of our knowledge, the following results are not known in the literature.
\begin{lemma}
\label{Thm:HammingGraphMulti}
For $k \leq q^n$, $\chi_k(H(n,q,\{ 1 \} )) = kq$.
\end{lemma}

\begin{proof}
Let us denote $H = H(n, q, \{1 \} )$. Consider the vectors in $H$ for which the first entry ranges from $0$ up to and including $q-1$, while the other entries equal $0$. This gives a clique of size $q$ and since $\omega(H) \leq \chi(H) = q$, we have $\omega(H) = q$. From \eqref{eqn:MultichromBounded}, it follows the result.
\end{proof}

The proof of lemma \ref{Thm:HammingGraphMulti} relies on the fact that $\omega(H) = \chi(H)$, or equivalently, that $H(n,q, \{1\})$ is a weakly perfect graph. In general, for any weakly perfect graph $G$
\begin{equation}
    \label{eqn:WeaklyPerfectMulticoloring}
    \omega(G) = \chi(G) \Longrightarrow \chi_k(G) = k \chi(G).
\end{equation}

This gives rise to the question for which values of $q$ and $n$ the graph $H(n,q, \{ 1 \} )$ is perfect. The strong perfect graph theorem states that a graph is perfect if and only if it does not contain $C_{2n+1}$ or $\overline{C}_{2n+1}$
as induced subgraphs, for all $n > 1$.
\begin{proposition}
\textit{The Hamming graph} $H(n, q, \{ 1 \} )$ \textit{is a perfect graph if and only if} $n \leq 2$ or $q \leq 2$.
\end{proposition}
\begin{proof}
Denote $H(n,q) = H(n,q,\{ 1 \})$. Graph $H(1,q)$ is $K_q$, which is clearly a perfect graph. Graph $H(2, q)$ is a lattice graph, or Rook's graph, which is also a perfect graph. Graph $H(n, 1)$ is a single vertex and thus also perfect. Lastly, graph $H(n, 2)$ is bipartite and thus perfect. For $q \geq 3$, the following vectors from $H(3,q)$ form $C_7$:
\begin{equation*}
    \begin{bmatrix}
    0 \\ 0 \\ 0
    \end{bmatrix},
    \begin{bmatrix}
    1 \\ 0 \\ 0
    \end{bmatrix},
    \begin{bmatrix}
    1 \\ 1 \\ 0
    \end{bmatrix},
    \begin{bmatrix}
    1 \\ 1 \\ 1
    \end{bmatrix},
    \begin{bmatrix}
    2 \\ 1 \\ 1
    \end{bmatrix},
    \begin{bmatrix}
    2 \\ 0 \\ 1
    \end{bmatrix},
    \begin{bmatrix}
    0 \\ 0 \\ 1
    \end{bmatrix}.
\end{equation*}
Then by the strong perfect graph theorem, $H(3, q)$ is not perfect. An odd cycle in $H(n,q)$ for general $n,q \geq 3$, is obtained by simply adjoining zeros to the above vectors such that they become $n$-dimensional.
\end{proof}
As $H(n,q, \{ f \} )$ is edgeless for $f > n$, we consider the extremal case $H(n, q, \{ n \} )$ for $n > 1$. Note that $H(1, q, \{ 1 \} )= K_q$. Graph $H(n,q,\{ n\})$ can be described by use of the tensor product of graphs (see definition \ref{Def:GraphProducts}).
In particular, we have  that $H(n, q, \{ n \} ) = \otimes ^n K_q$.

Since all the edges of $G_1 \otimes G_2$ also appear in $G_1 \circ G_2$, it follows that  $\chi(G_1 \otimes G_2) \leq \chi(G_1 \circ G_2)$. Moreover, by  \citeauthor{hedetniemi1966homomorphisms} \cite{hedetniemi1966homomorphisms}, we have
\begin{equation}
    \label{eqn:ChromaticTensorInequality}
  \chi(G_1 \otimes G_2) \leq \min \{ \chi(G_1), \chi(G_2) \}.
\end{equation}
\textit{Hedetniemi's conjecture} states that \eqref{eqn:ChromaticTensorInequality} holds with equality. The conjecture was recently  disproved by \citeauthor{shitov2019counterexamples} \cite{shitov2019counterexamples}.
Inequality \eqref{eqn:ChromaticTensorInequality} implies  that $\chi(\otimes^n K_q) \leq q$.

The vectors $i \cdot \vec{1}$ for $ 0 \leq i \leq q-1$ form a clique of size $q$ in graph $H(n,q, \{ n \} )$. Thus $q\leq \omega(H(n,q,\{n \}))$. Now, from this inequality and  $\chi( \otimes^n K_q ) \leq q$ it follows that $\chi(H(n,q, \{n \} )) = q$. Using \eqref{eqn:MultichromBounded} or \eqref{eqn:WeaklyPerfectMulticoloring}, we find $\chi_k(H(n,q,\{ n\} )) = kq$.  The coloring of these tensor products of graphs has been previously considered by \citeauthor{greenwell1974applications} \cite{greenwell1974applications}, where they also proved this result.

Let us now define $\Setname := \{ i \in \mathbb{N} \, | \, f \leq i \, \}.$ The Hamming graph $H(n,q, \Setname)$ has been studied by \citeauthor{el2007bounds} \cite{el2007bounds} among others. They show that, under some condition on the parameters $n$, $q$ and $f$, $\chi(H(n,q, \Setname)) = q^{n-f+1}$. We extend this result to multicoloring in the following proposition.

\begin{proposition} \label{prop:chiHamming}
For $q \geq n-f+2$ and $1 \leq f \leq n$, we have $\chi_k(H(n,q, \Setname)) = kq^{n-f+1}.$
\end{proposition}
\begin{proof}
For parameters $n$, $q$ and $f$ satisfying the conditions of the proposition, it is known (cf. \cite{frankl1999erdHos}) that $\alpha(H(n,q, \Setname)) = q^{f-1}$. By \eqref{eqn:MultichromBounded} and \eqref{eqn:SimpleMulticoloringVerticesLB}, the proposition follows.
\end{proof}
Binary Hamming graphs $H(n,2,\{ f\})$, $f \leq n$ form another interesting case.
Recall that the Hamming weight of a vector is its Hamming distance to the zero vector, and that
the Hamming graphs are vertex-transitive.
\begin{theorem}
\label{Thm:GenHamming}
\textit{For all $n \in \mathbb{N}$, $f$ odd and $f \leq n$, } $\chi_k(H(n,2,\{f \})) = 2k.$
\end{theorem}
\begin{proof}
Let $n,f \in \mathbb{N}$, $f$ odd and $f \leq n$. Consider the zero vector in $H(n,2,\{f \})$. Note that every vector adjacent (orthogonal) to the zero vector has an odd Hamming weight. By vertex transitivity, all vectors of even Hamming weight only have vectors of odd Hamming weight as neighbors. Similarly, vectors of odd Hamming weight only have vectors of even Hamming weight as neighbors. Graph $H(n,2,\{f \})$ is thus bipartite, which, combined with \eqref{eqn:MultichromBounded}, proves the theorem.
\end{proof}
It readily follows the following results for  orthogonality graphs.
\begin{corollary}
\textit{Let $\Omega_{4n+2}$ ($n \in \mathbb{N}$) be the orthogonality graph. Then,
$\chi_k(\Omega_{4n+2}) = 2k$.}
\end{corollary}
\begin{proof}
Graph $\Omega_{4n+2}$ is isomorphic to $H(4n+2,2, \{2n+1\})$. This corollary is thus a special case of theorem \ref{Thm:GenHamming}.
\end{proof}

\section{Conclusion}

In this paper, we study the generalized $\vartheta$-number for highly symmetric graphs and beyond. The parameter $\vartheta_k(G)$   generalizes the concept of the  famous $\vartheta$-number  that  was introduced by   \citeauthor{lovasz1979shannon} \cite{lovasz1979shannon}.
Since $\vartheta_k(G)$ is sandwiched between the  $\alpha_k(G)$ and $\chi_k (\overline{G})$ it serves as a bound for both graph parameters.

Several results in this paper are not restricted to highly symmetric graphs.
In particular, the results in section \ref{section:formulations}, section \ref{sect:the_series} and section \ref{sect:Graph products}.
In section \ref{section:formulations} we present in an elegant way a known result that  $\vartheta_k (G)$ is a lower bound for  $\chi_k (\overline{G})$. Another lower bound for  $\chi_k (\overline{G})$
is $k\vartheta(G)$, see \eqref{eqn:VarthetaIneq}.
The inequality \eqref{eqn:VarthetaIneq} is rather  counter-intuitive since it is more difficult to compute $\vartheta_k(G)$ than  $\vartheta(G)$, while  $k\vartheta(G)$ provides a better bound for the $k$-th chromatic number.
However,  the generalized  $\vartheta$-number can also be used to compute lower bounds for the (classical) chromatic number of a graph, see section \ref{Section:AnalysisOfCart}.

In section \ref{sect:the_series} we show that the sequence  $(\vartheta_k(G))_k$
is increasing  and bounded above by the order of $G$ (proposition \ref{thm:ThetaKIncreasing} and theorem \ref{thm:SecondOrderDiff}),
and that the increments of the sequence can be arbitrarily small (theorem \ref{thm:DeltaKGraphName}). Section \ref{sect:Graph products}  provides bounds for $\vartheta_k(G)$ where $G$ is the strong graph product of two graphs (theorem \ref{Thm:StrongProdThetaK})  and the disjunction  product of two graphs (theorem \ref{thm:GraphProdDisjunction}).

Sections \ref{Section:ValueOfThetaKForSomeGraphs}, \ref{section:SRG} and  \ref{section:OrthogonalityGraph} consider highly symmetric graphs. We derive closed form expressions for the generalized $\vartheta$-number on  cycles (theorem \ref{Thm:OddCycleTheta2}),  Kneser graphs (theorem \ref{Thm:KneserGraphThetak}), Johnson graphs (theorem \ref{Thm:JhonsonGraphThetak}),
strongly regular graphs (theorem \ref{THm:StronglyRegularThetak}), among other results.
It is known that  $\vartheta(K_k \square G)$  and  $\vartheta_k(G)$ provide upper bounds on $\alpha_k(G)$.  However, it is more computationally demanding to compute  $\vartheta(K_k \square G)$ than  $\vartheta_k(G)$. We show that for graphs that are both edge-transitive and vertex-transitive   it  suffices to solve $\vartheta_k(G)$, see  theorem \ref{thm:CartesianThetaK}.
However, the gap between  $\vartheta_k(G)$ and $\vartheta(K_k \square G)$ can be  arbitrarily large (proposition \ref{prop:Large gap}). We also prove that  $\vartheta'(K_k \square G)$ equals $\vartheta'_k(G)$  for any (non-trivial) strongly regular graph $G$  and $k<n(r+1)/(r+n-d)$, see theorem \ref{thm:CartesianThetaKSRG}. Section \ref{section:OrthogonalityGraph}  presents results for $\vartheta_k(G)$ and $\chi_k(G)$ on orthogonality graphs.

Bounds on the $k$-th chromatic number of various graphs are given in section \ref{sect:lower_bound_for_multicoloring}.
In particular,  bounds on the product and sum of $\chi_k({G})$ and $\chi_k(\overline{G})$ are presented in theorem \ref{Thm:NordhausMulticoloring}. Lemma~\ref{Thm:HammingGraphMulti}, proposition \ref{prop:chiHamming} and theorem~\ref{Thm:GenHamming} provide the multichromatic number for several Hamming graphs,   while proposition \ref{prop:chiTriangular} provides bounds for the multichromatic number on  triangular graphs.

Let us list several open problems.
Prove conjecture \ref{conj:SDPinequalities} for any graph. Recall that we prove  conjecture \ref{conj:SDPinequalities} only for symmetric graphs, see theorem \ref{thm:CartesianThetaK}.
It would be interesting to prove the  conjecture by Godsil, Newman and Frankl \eqref{conjecture} for the first open case  $\Omega_{40}$. Another open problem is to generalize the well-known inequality $\vartheta(G)\vartheta(\overline{G}) \geq |V|$, see~\cite{lovasz1979shannon}, for $\vartheta_k(G)$, $k\geq 2$.

\medskip\medskip

\noindent
{\bf Acknowledgements}
We would like to  thank Ferdinand Ihringer  for pointing us to reference  \cite{frankl1986} and the result $\alpha(\Omega_{24}) = 178208$, where $\Omega_{24}$  is the orthogonality graph.

\medskip\medskip

\bibliographystyle{abbrvnat}
\bibliography{bibMulticolor}

\end{document}